\documentclass[preprint,10pt]{elsarticle}
\usepackage{amssymb}
\usepackage{atbegshi}
\AtBeginDocument{\AtBeginShipoutNext{\AtBeginShipoutDiscard}}
\usepackage[english,french]{babel}
\usepackage[utf8]{inputenc} 
\usepackage{times}
\usepackage{subcaption}
\usepackage{amsthm}
\usepackage{lmodern}
\usepackage[a4paper]{geometry}
\usepackage{epstopdf}
\usepackage{mathrsfs}
\usepackage{latexsym,amssymb,amsfonts}
\usepackage{enumerate}
\usepackage{graphicx}
 \usepackage{url}

\usepackage{float}
\usepackage{graphicx}

\usepackage{alltt}
\usepackage{amsfonts}
\usepackage{amsmath}
\usepackage[labelfont=bf]{caption}
\usepackage{hyperref}
\hypersetup{
    colorlinks,
    citecolor=red,
    filecolor=red,
    linkcolor=red,
    urlcolor=red
}

\newtheorem{thm}{Theorem}[section]

\newtheorem{lem}[thm]{Lemma}

\theoremstyle{definition}

\newtheorem{rem}[thm]{Remark}
\newtheorem{ex}[thm]{Example}

\def\R{\mathbb{R}}

\def\P{\mathbb P}

\def\E{\mathbb E}

\usepackage[toc,page]{appendix}
\journal{*}
\usepackage{geometry}
\begin{document}
\selectlanguage{english}
\begin{frontmatter}

\title{New results on the asymptotic behavior of an  SIS epidemiological model with quarantine strategy, stochastic transmission, and Lévy disturbance}
\author{Driss Kiouach,\footnote{Corresponding author.\\
E-mail addresses: \href{d.kiouach@uiz.ac.ma}{d.kiouach@uiz.ac.ma} (D. Kiouach), \href{yassine.sabbar@edu.uiz.ac.ma}{yassine.sabbar@edu.uiz.ac.ma} (Y. Sabbar), \href{azamisalim9@gmail.com}{azamisalim9@gmail.com} (S. El Azami El-idrissi).} Yassine Sabbar and Salim El Azami El-idrissi}
\address{LPAIS Laboratory, Faculty of Sciences Dhar El Mahraz, Sidi Mohamed Ben Abdellah University, Fez, Morocco.}   
\vspace*{1cm}

\begin{abstract}
The spread of infectious diseases is a major challenge in our contemporary world, especially after the recent outbreak of Coronavirus disease 2019 (COVID-19). The quarantine strategy is one of the important intervention measures to control the spread of an epidemic by greatly minimizing the likelihood of contact between infected and susceptible individuals. In this study, we analyze the impact of various stochastic disturbances on the epidemic dynamics during the quarantine period. For this purpose, we present an SIQS epidemic model that incorporates the stochastic transmission and the Lévy noise in order to simulate both small and massive perturbations. Under appropriate conditions, some interesting asymptotic properties are proved, namely: ergodicity, persistence in the mean, and extinction of the disease. The theoretical results show that the dynamics of the perturbed model are determined by parameters that are closely related to the stochastic noises. Our work improves many existing studies in the field of mathematical epidemiology and provides new techniques to predict and analyze the dynamic behavior of epidemics. \vskip 2mm


\end{abstract}
\end{frontmatter}
\
\nocite{2009ProcDETAp}
\section{Introduction}
The study of infectious diseases has long been a subject where epidemiological issues are combined with financial and social problems \cite{1,2,3,4,5}. The rapid spread of COVID-19 these days shows that humanity stills suffer from epidemics that may lead to the collapse of medical and economic systems. By isolating infected individuals and quarantining the susceptible population at home, many countries have basically controlled the outbreak of COVID-19 \cite{Q1,Q2}. In order to analyze the impact of this strategy on the spread of epidemics and to predict their future behavior, we use different mathematical formulations according to their characteristics \cite{sis6,sis7,sis9}. In this study, we consider an SIQS epidemic model in the form of ordinary differential equations (ODEs for short). These ODEs describe the evolution of susceptible $S(t)$, infected $I(t)$, and isolated $Q(t)$ individuals as time functions. The rates of change and the interactions between different population classes in our case are expressed by the following deterministic model \cite{sis3}:
\begin{align}
\begin{cases}
\textup{d}S(t)=\big(A -\mu_1 S(t)-\beta S(t) I(t)+\gamma I(t)+k Q(t)\big)\textup{d}t,\\
\textup{d}I(t)=\big(\beta I(t) S(t) -(\mu_1+r_2+\delta +\gamma)I(t)\big)\textup{d}t,\\
\textup{d}Q(t)=\big(\delta I(t)-(\mu_1+r_3+k)Q(t)\big)\textup{d}t,
\end{cases}
\label{s1}
\end{align}
where the parameters appearing in this system are described as follows:
\begin{itemize}
\item[$\bullet$] $A$ is the recruitment rate of the susceptible individuals corresponding to new births.
\item[$\bullet$] $\mu_1$ is the natural death rate.
\item[$\bullet$] $\delta$ is the isolation rate.
\item[$\bullet$] $r_2$ is the disease-related mortality rate.
\item[$\bullet$] $r_3$ is the death rate associated with the disease under isolation intervention. For simplicity, we denote $\mu_2=\mu_1+r_2$ and $\mu_3=\mu_1+r_3$ as a general mortality rates.
\item[$\bullet$] $\gamma$ and $k$ are the rates which individuals recover and return to $S$ from $I$ and $Q$, respectively.
\item[$\bullet$] $\beta$ represents the transmission rate. 
\end{itemize} 
\begin{figure}[hbtp]
\begin{center}
\includegraphics[scale=0.35]{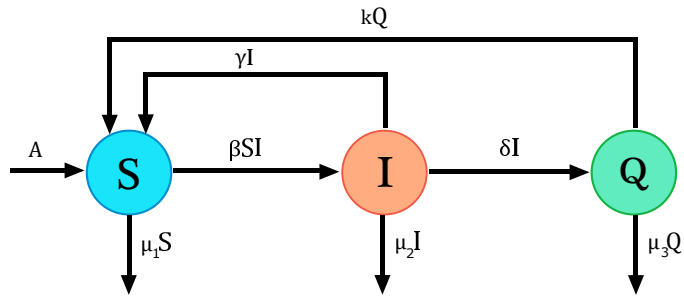} 
\end{center} 
\caption{The transfer diagram for the deterministic SIQS epidemic model (\ref{s1}).}
\label{fig0}
\end{figure}
All parameters are usually assumed to be positive. The schematic flow diagram of the model (\ref{s1}) is illustrated in Figure \ref{fig0}. Herbert et al. \cite{sis3} proved that the basic reproduction number of the deterministic model (\ref{s1}) is expressed by $\mathcal{R}_0=\frac{\beta A}{\mu_1 (\mu_2+\delta+\gamma)}$. This parameter is an essential quantity to predict whether a disease will persist or not. If $\mathcal{R}_0\leq 1$, the model (\ref{s1}) has only the disease-free equilibrium $E^{\circleddash}=(A/\mu_1,0,0)$ which is globally asymptotically stable, and if $\mathcal{R}_0>1$, $E^{\circleddash}$ becomes unstable and there exists a global asymptotically stable endemic equilibrium $E^{\circledast}=(S^{\circledast},I^{\circledast},Q^{\circledast})$, where
\begin{align*}
S^{\circledast}=\frac{A}{\mu_1\mathcal{R}_0},\hspace*{0.5cm}I^{\circledast}=\frac{A(1-1/\mathcal{R}_0)}{\mu_2(1+\delta/(\mu_3+k))}, \hspace*{0.2cm}\text{and}\hspace*{0.2cm} Q^{\circledast}=\frac{\delta I^{\circledast}}{(\mu_3+k)}.
\end{align*}

The spread of infectious diseases can undergo random disturbances and stochastic phenomena due to environmental fluctuations \cite{9,55,54}. Since disturbed models can describe many practical problems very well, many types of stochastic differential equations have been used to analyze various epidemic models in recent years \cite{53,52,51}. There are two common ways to introduce stochastic factors into epidemic systems. The first one is to assume that the transmission of the diseases is subject to some small random fluctuations which can be described by the Gaussian white noise \cite{sis11,sis12,sis13}. The other one is to admit that the model parameters are affected by massive environmental perturbations like the climate changes, earthquakes, hurricanes, floods, etc \cite{56,57}. For a better explaination to these phenomena, the use of a compensated Poisson process into the population dynamics provides an appropriate and more realistic context \cite{sis14}. Considering these two types of random disturbances, many works have analyzed the asymptotic behaviors of various epidemic models, including persistence in the mean, extinction, and ergodicity \cite{58,56,57,sis1,sis2}. These interesting researches have served an important role in the stochastic modeling of epidemics. But, all these models have considered either the standard white Gaussian noise or the Lévy jumps. In this work, we combine these two perturbations by treating an SIQS epidemic model that simultaneously includes the stochastic transmission and the discontinuous Lévy process. This original idea extends the studies presented in \cite{sis2,sis14} and gives us a general view of the disease dynamics under different scenarios of random perturbations. \\

The threshold analysis of perturbed epidemic systems is very important for understanding and controlling of the disease spread. In our case, the deterministic model (\ref{s1}) will be perturbed not only by white noise but also by Lévy jumps, which makes its analysis more complicated and needs some new techniques and methods. During this study, we aim to develop a mathematical approach to prove the existence of a unique ergodic stationary distribution and persistence in the mean of the new perturbed model. Without using the classical Lyapunov method presented in \cite{12}, we obtain sufficient conditions for the ergodicity by employing the Feller property and mutually exclusive possibilities lemma. Under the same conditions, we demonstrate that the persistence in the mean of the disease occurs. To analyze properly our new model, we study the stochastic extinction case. \\

The organization of this paper is as follows: in section \ref{sec2},  we present our new stochastic system and some preliminary results. In section \ref{sec3}, we focus on the stochastic characteristics of the perturbed model. Since the ergodicity is an important statistical characteristic, the existence of a unique stationary distribution is obtained. Almost sufficient condition for the persistence is also established. To complete our
study, we give sufficient conditions for the disease extinction. Finally, we support our theoretical results by illustrating some numerical examples.
\section{The stochastic SIQS model and some preliminaries}\label{sec2}
Let $(\Omega,\mathcal{F},\P)$ be a complete probability space  with a filtration $\{\mathcal{F}_t\}_{t\geq 0}$ satisfying the usual conditions, and containing all the random variables that will be meted in this paper. We merge the stochastic transmission with a discontinuous perturbed mortality rates. The random variability in the epidemic transmission $\beta$ and the mortality rates $\mu_i$ $(i=1,2,3)$ are presented by a decomposition of usual white noise and the Lévy-Itô process, respectively. Under these assumptions, the evolution of an epidemic during the quarantine strategy is modeled by the following system of stochastic differential equations: 
\begin{align}
\begin{cases}
\textup{d}S(t)=\big(A -\mu_1 S(t)-\beta S(t) I(t) +\gamma I(t)+ k Q(t)  \big) \textup{d}t+\mathcal{P}_1(t),\\
\textup{d}I(t)=\big( \beta S(t) I(t) -( \mu_2 +\delta +\gamma)I(t) \big) \textup{d}t+\mathcal{P}_2(t) ,\\
\textup{d}Q(t)=\big( \delta I(t) - (\mu_3 + k) Q(t) \big)\textup{d}t+\mathcal{P}_3(t),
\end{cases}
\label{s2}
\end{align}
where
\begin{align*}
\mathcal{P}_1(t)&=\sigma_1 S(t) \textup{d}\mathcal{W}_1(t)+\int_Z \eta_1(u)S(t^{-})\widetilde{\mathcal{N}}(\textup{d}t,\textup{d}u)- \sigma_{\beta} S(t)I(t) \textup{d}\mathcal{W}_{\beta}(t),\\
\mathcal{P}_2(t)&=\sigma_2 I(t) \textup{d}\mathcal{W}_2(t)+\int_Z \eta_2(u)I(t^{-})\widetilde{\mathcal{N}}(\textup{d}t,\textup{d}u)+ \sigma_{\beta} S(t)I(t) \textup{d}\mathcal{W}_{\beta}(t),\\
\mathcal{P}_3(t)&=\sigma_3 Q(t) \textup{d}\mathcal{W}_3(t)+\int_Z \eta_3(u)Q(t^{-})\widetilde{\mathcal{N}}(\textup{d}t,\textup{d}u).
\end{align*}
Here, $\mathcal{W}_{\beta}(t)$ and $\mathcal{W}_{i}(t)$ $(i = 1, 2,3)$ are the mutually independent Brownian motions defined on $(\Omega,\mathcal{F},\{\mathcal{F}_t\}_{t\geq 0},\P)$ with the positive intensities $\sigma_{\beta}$ and $\sigma_{i}$ $(i = 1, 2,3)$. $S(t^{-})$, $I(t^{-})$ and $Q(t^{-})$ are the left limits of $S(t)$, $I(t)$ and $Q(t)$, respectively. $\mathcal{N}$ is a Poisson counting measure with compensating  martingale $\widetilde{\mathcal{N}}$ and characteristic measure $\nu$ on a measurable subset $Z$ of $(0,\infty)$ satisfying $\nu(Z)<\infty$. It assumed that $\nu$ is a Lévy measure such that $\widetilde{\mathcal{N}}(\textup{d}t,\textup{d}u)=\mathcal{N}(\textup{d}t,\textup{d}u)-\nu(\textup{d}u)\textup{d}t$. We also assume that $\mathcal{W}_{i}(t)$ $(i = 1, 2,3,\beta)$ are independent of $\mathcal{N}$. The bounded functions $\eta_i:\; Z\times\Omega\to\R$ $(i = 1, 2,3)$ are $\mathfrak{B}(Z)\times\mathcal{F}_t$-measurable and  continuous with respect to $\nu$.\\
For the sake of notational simplicity, we define
\vspace{0.1cm}
\begin{itemize}
\item[$\bullet$] $\bar{\sigma}=\max\{\sigma_1^2,\sigma_2^2,\sigma_3^2\}$.
\item[$\bullet$] $\bar{\eta}(u)=\max\{\eta_1(u),\eta_2(u),\eta_3(u)\}$.
\item[$\bullet$] $\underline{\eta}(u)=\min\{\eta_1(u),\eta_2(u),\eta_3(u)\}$.
\item[$\bullet$] $\hat{\rho}_{n,p}(u)=\big[1+\bar{\eta}(u)\big]^{np}-1-np\bar{\eta}(u)$.
\item[$\bullet$] $\check{\rho}_{n,p}(u)=\big[1+\underline{\eta}(u)\big]^{np}-1-np\underline{\eta}(u)$.
\item[$\bullet$] $\ell_{n,p}=\int_Z\big[\hat{\rho}_{n,p}(u)\vee\check{\rho}_{n,p}(u)\big]\nu(\textup{d}u)$.
\end{itemize}
\vspace{0.2cm}
To properly study our model (\ref{s2}), we have the following fundamental assumptions on the jump-diffusion coefficients:
\vspace*{5pt}
\begin{itemize}
\item[$\bullet$] \textbf{($\textup{A}_1$)} We assume that the jump coefficients $\eta_i(u)$ in (\ref{s2}) satisfy $\int_Z \eta_i^2(u)\nu(\textup{d}u)<\infty$, $\{i=1,2,3\}$.
\item[$\bullet$] \textbf{($\textup{A}_2$)} For all $u\in Z$, we assume that $1+\eta_i(u)>0$ and $\int_Z\big[\eta_i(u)-\ln (1+\eta_i(u))\big]\nu(\textup{d}u)<\infty,$ $\{i=1,2,3\}$.
\item[$\bullet$] \textbf{($\textup{A}_3$)} We suppose that $\int_Z \big[\ln(1+\eta_i(u))\big]^2\nu(\textup{d}u)<\infty,$ $\{i=1,2,3\}$.
\item[$\bullet$] \textbf{($\textup{A}_4$)} We suppose that $\int_Z \big[\big(1+\bar{\eta}(u)\big)^{2}-1\big]^2\nu(\textup{d}u)<\infty$.
\item[$\bullet$] \textbf{($\textup{A}_5$)} We suppose that for each positive integer $n$ there is some real number $p>1$ for which $$\Gamma_{n,p}=\mu_1-\frac{(np-1)}{2}\bar{\sigma}-\frac{1}{np}\ell_{n,p}> 0.$$ 
\end{itemize}
\vspace*{-4pt}
\par In view of the biological interpretation, the question of whether the stochastic model is well-posed is the first concern. Therefore, to analyze the stochastic model (\ref{s2}), it is necessary to verify the existence of a unique global positive solution, that is, there is no explosion in finite time for any positive initial value $(S(0), I(0), Q(0))\in \R^{3}_{+}$. The following lemma assures the well-posedness of the stochastic model (\ref{s2}).
\begin{lem}
Let assumptions \textbf{($\textup{A}_1$)} and \textbf{($\textup{A}_2$)} hold. For any initial value $Y(0)=(S(0),I(0),Q(0))\in \R^3_{+}$, there exists a unique positive solution $Y(t)=(S(t),I(t),Q(t))$ of system (\ref{s2}) on $t\geq 0$, and this solution will stay in $\R^{3}_{+}$ almost surely. 
\label{thmp}
\end{lem}
The proof is somehow standard and classic (see for example \cite{16,166}), so we omit it here. \\[5pt]
In the following, we always presume that the assumptions \textbf{$(\mbox{A}_1)$} - \textbf{$(\mbox{A}_5)$} hold. For reference purposes, we will prepare several useful lemmas.
\begin{lem}
Let $n$ be a positive integer and let $Y(t)$ denotes the solution of system (\ref{s2}) that starts from a given point $Y(0)\in\R^3_+$. Then, for any $p>1$ that satisfies $\Gamma_{n,p}>0$, we have 
\begin{enumerate}
\item[$\bullet$] $\displaystyle{\E\big[N^{np}(t)\big]\leq N^{np}(0)\times e^{-\frac{np\Gamma_{n,p}}{2}t}+ \frac{2\Delta}{\Gamma_{n,p}}}$.
\item[$\bullet$] $\displaystyle{\underset{t\to \infty}{\limsup} \frac{1}{t}\int^t_0\E\big[N^{np}(s)\big]\textup{d}s\leq \frac{2\Delta}{\Gamma_{n,p}}}$\;\; a.s.
\end{enumerate}
where $\Delta=\underset{N>0}{\sup}\{AN^{{np}-1}-\frac{ \Gamma_{n,p}}{2}N^{np}\}$ and $N(t)=S(t)+I(t)+Q(t)$.
\label{L1}
\end{lem}
\begin{proof}
Making use of Itô's lemma \cite{lev1} to $N^{np}(t)$, we obtain
\begin{align*}
\textup{d}N^{np}(t)&=\bigg\{npN^{np-1}(t)\Big(A-\mu_1 N(t)-r_2I(t)-r_3Q(t)\Big)+\dfrac{np}{2}(np-1) N^{np-2}(t)\Big(\sigma_1^2S^2(t)+\sigma_2^2I^2(t)+\sigma_3^2Q^2(t)\Big)\\
&\;\;\;+\int_ZN^{np}(t)\bigg[\bigg(1+\eta_1(u)\frac{S(t)}{N(t)}+\eta_2(u)\frac{I(t)}{N(t)}+\eta_3(u)\frac{Q(t)}{N(t)}\bigg)^{np}-1\\
&\;\;\;-np\bigg(\eta_1(u)\frac{S(t)}{N(t)}+\eta_2(u)\frac{I(t)}{N(t)}+\eta_3(u)\frac{Q(t)}{N(t)}\bigg)\bigg]\nu(\textup{d}u)\bigg\}\textup{d}t\\
&\;\;\;+npN^{np-1}(t)\Big(\sigma_1 S(t)\textup{d}\mathcal{W}_1(t)+\sigma_2 I(t)\textup{d}\mathcal{W}_2(t)+\sigma_3 Q(t)\textup{d}\mathcal{W}_3(t)\Big)\\
&\;\;\;+\int_Z N^{np}(t^{-})\bigg[\bigg(1+\eta_1(u)\frac{S(t^{-})}{N(t^{-})}+\eta_2(u)\frac{I(t^{-})}{N(t^{-})}+\eta_3(u)\frac{Q(t^{-})}{N(t^{-})}\bigg)^{np}-1\bigg]\widetilde{\mathcal{N}}(\textup{d}t,\textup{d}u).
\end{align*}
Then
\begin{align}
\textup{d}N^{np}(t)&\leq\bigg(npN^{np-1}(t)\Big(A-\mu_1 N(t)\Big)+\dfrac{np}{2}(np-1)N^{np}(t)\bar{\sigma}+N^{np}(t)\int_Z\big[\hat{\rho}_{n,p}(u)\vee\check{\rho}_{n,p}(u)\big]\nu(\textup{d}u)\bigg)\textup{d}t\nonumber\\
&\;\;\;+npN^{np-1}(t)\Big(\sigma_1 S(t)\textup{d}\mathcal{W}_1(t)+\sigma_2 I(t)\textup{d}\mathcal{W}_2(t)+\sigma_3 Q(t)\textup{d}\mathcal{W}_3(t)\Big)\nonumber\\
&\;\;\;+\int_Z N^{np}(t^{-})\Big((1+\bar{\eta}(u))^{np}-1\Big)\widetilde{\mathcal{N}}(\textup{d}t,\textup{d}u).
\label{maj}
\end{align}
Rewriting the above inequality, one can see that
\begin{align*}
\textup{d}N^{np}(t)&\leq np\Bigg\{AN^{np-1}(t)-\bigg(\mu_1-\frac{(np-1)}{2}\bar{\sigma}-\frac{1}{np}\int_Z\big[\hat{\rho}_{n,p}(u)\vee\check{\rho}_{n,p}(u)\big]\nu(\textup{d}u)\bigg)N^{np}(t)\Bigg\}\textup{d}t\\
&\;\;\;+npN^{np-1}(t)\Big(\sigma_1 S(t)\textup{d}\mathcal{W}_1(t)+\sigma_2 I(t)\textup{d}\mathcal{W}_2(t)+\sigma_3 Q(t)\textup{d}\mathcal{W}_3(t)\Big)\\&\;\;\;+\int_Z N^{np}(t^{-})\Big((1+\bar{\eta}(u))^{np}-1\Big)\widetilde{\mathcal{N}}(\textup{d}t,\textup{d}u).
\end{align*}
We choose neatly $p>1$ such that $\displaystyle{\Gamma_{n,p}=\mu_1-\frac{(np-1)}{2}\bar{\sigma}-\frac{1}{np}\int_Z\big[\hat{\rho}_{n,p}(u)\vee\check{\rho}_{n,p}(u)\big]\nu(\textup{d}u)> 0}$. Therefore 
\begin{align*}
\textup{d}N^{np}(t)&\leq np\Big\{\Delta-\frac{\Gamma_{n,p}}{2}N^{np}(t)\Big\}\textup{d}t+npN^{np-1}(t)\Big(\sigma_1 S(t)\textup{d}\mathcal{W}_1(t)+\sigma_2 I(t)\textup{d}\mathcal{W}_2(t)+\sigma_3 Q(t)\textup{d}\mathcal{W}_3(t)\Big)\\
&\;\;\;+\int_Z N^{np}(t^{-})\Big((1+\bar{\eta}(u))^{np}-1\Big)\widetilde{\mathcal{N}}(\textup{d}t,\textup{d}u).
\end{align*}
On the other hand, we have
\begin{align*}
\text{d} N^{np}(t)\times e^{\frac{np\Gamma_{n,p}}{2}t }&=p\Gamma_{n,p}~N^{np}(t)\times e^{\frac{np\Gamma_{n,p}}{2}t }+e^{\frac{np\Gamma_{n,p}}{2}t }\textup{d}N^{np}(t)\\
&\leq np\Delta e^{\frac{np\Gamma_{n,p}}{2}t }+e^{\frac{np\Gamma_{n,p}}{2}t }\bigg[npN^{np-1}(t)\Big(\sigma_1 S(t)\textup{d}\mathcal{W}_1(t)+\sigma_2 I(t)\textup{d}\mathcal{W}_2(t)+\sigma_3 Q(t)\textup{d}\mathcal{W}_3(t)\Big)\\&\;\;\;+\int_Z N^{np}(t^{-})\Big((1+\bar{\eta}(u))^{np}-1\Big)\widetilde{\mathcal{N}}(\textup{d}t,\textup{d}u)\bigg].
\end{align*}
Then, by taking the integration and the expectations, we get
\begin{align*}
\E \big[N^{np}(t)\big]&\leq N^{np}(0)\times e^{-\frac{np\Gamma_{n,p}}{2}t }+np\Delta\int^t_0 e^{-\frac{np}{2}\Gamma_{n,p}(t-s)}\textup{d}s\leq N^{np}(0)e^{-\frac{np\Gamma_{n,p}}{2}t }+\frac{2\Delta}{\Gamma_{n,p}}.
\end{align*}
Obviously, we obtain
\begin{align*}
\underset{t\to \infty}{\lim \sup} \frac{1}{t}\int^t_0\E\big[N^{np}(s)\big]\textup{d}s\leq N^{np}(0)\times\underset{t\to \infty}{\lim \sup} \frac{1}{t}\int^t_0e^{-\frac{np\Gamma_{n,p}}{2}s}\textup{d}s+\frac{2\Delta}{\Gamma_{n,p}}=\frac{2\Delta}{\Gamma_{n,p}}.
\end{align*}
This completes the proof.
\end{proof}
\begin{rem}
Throughout this remark, $\tilde{X}$ is standing for the sum $\eta_1(u)S+\eta_2(u)I+\eta_3(u)Q$, where $u\in Z$. In the study of stochastic biological models driven by Lévy jumps (see for example, \cite{57,166,lev1,lev2,lev3}), the following quantity
\begin{align*}
\int_ZN^{np}\left[\Big(1+\frac{\tilde{X}}{N}\Big)^{np}-1-np\frac{\tilde{X}}{N}\right]\nu(\textup{d}u),
\end{align*}
is widely majorazed by
\begin{align*}
\int_ZN^{np}\Big((1+\bar{\eta}(u))^{np}-1-np\underline{\eta}(u)\Big)\nu(\textup{d}u).
\end{align*}
However, the last estimation can be ameliorated by considering the following inequality
\begin{align}
\int_ZN^{np}\left[\Big(1+\frac{\tilde{X}}{N}\Big)^{np}-1-np\frac{\tilde{X}}{N}\right]\nu(\textup{d}u)\leq \int_ZN^{np}\big[\hat{\rho}_{n,p}(u)\vee\check{\rho}_{n,p}(u)\big]\nu(\textup{d}u),
\label{new9}
\end{align}
which is established from the observation that the function 
\begin{align*}
g(x)=(1+x)^{np}-1-npx,\hspace{0.3cm}n,p\geq 1,
\end{align*}
is decreasing for $x\in(-1,0)$ and increasing for $x \geq 0$. Needless to say, this last fact  makes necessarily $g(a)\vee g(b)$ as the highest value of $g$ on any interval $[a,b]\subset(-1,\infty)$. The adoption of the inequality \eqref{new9} in our calculus, especially in (\ref{maj}), \eqref{est9} and \eqref{maj2}, will improve many classical results presented in the above mentioned papers.
\label{remm}
\end{rem}
\begin{rem}
Lemma \ref{L1} takes into consideration the stochastic transmission and the effect of Lévy jumps, and this makes it clearly an extended version of Lemma 2.3 presented in \cite{sharp}.  
\end{rem}
\begin{lem}
Consider the initial value problem
\begin{align}
\begin{cases}
\textup{d}X(t)=\big(A-\mu_1 X(t)\big)\textup{d}t+\bar{\mathcal{P}}_1(t)+\bar{\mathcal{P}}_2(t)+\mathcal{P}_3(t),\\
 X(0)=N(0)\in\mathbb{R}_+ ,
\end{cases}
\label{s4}
\end{align}  
where
\begin{align*}
\bar{\mathcal{P}}_1(t)&=\sigma_1 S(t) \textup{d}\mathcal{W}_1(t)+\int_Z \eta_1(u)S(t^{-})\widetilde{\mathcal{N}}(\textup{d}t,\textup{d}u),\\
\bar{\mathcal{P}}_2(t)&=\sigma_2 I(t) \textup{d}\mathcal{W}_2(t)+\int_Z \eta_2(u)I(t^{-})\widetilde{\mathcal{N}}(\textup{d}t,\textup{d}u).
\end{align*}
Let us denote by $X(t)$ and $Y(t)$ the positive solutions of systems (\ref{s2}) and (\ref{s4}) respectively. Then
\begin{enumerate}
\item[$\bullet$] $\displaystyle{\underset{t\to\infty}{\lim}\frac{X^n(t)}{t}=0 \hspace{0.2cm}\mbox{a.s.}},\quad \forall n\in \{1,2,\cdots\}.$
\item[$\bullet$] $\displaystyle{\underset{t\to\infty}{\lim}\frac{\int^t_0 X(s)S(s)\textup{d}\mathcal{W}_1(s)}{t}=0, \hspace{0.2cm} \underset{t\to\infty}{\lim}\frac{\int^t_0 X(s)I(s)\textup{d}\mathcal{W}_2(s)}{t}=0,  \hspace{0.2cm}\mbox{and}\hspace{0.2cm} \underset{t\to\infty}{\lim}\frac{\int^t_0 X(s)Q(s)\textup{d}\mathcal{W}_3(s)}{t}=0 \hspace{0.3cm}\mbox{a.s.}}$
\item[$\bullet$] $\displaystyle{\underset{t\to\infty}{\lim}\frac{\int^t_0\int_Z\big((1+\bar{\eta}(u))^2-1\big)X^2(s^{-})\widetilde{\mathcal{N}}(\textup{d}s,\textup{d}u)}{ t}=0 \hspace{0.3cm}\mbox{a.s.}}$
\end{enumerate}
\label{lem22}
\end{lem}
\begin{proof}
Our approach to demonstrate this lemma is mainly adapted from \cite{166}. The proof falls naturally into three steps.\\

\textbf{Step 1.}\;\; Applying the generalized Itô's formula \cite{lev1} to $\mathcal{K}(X)=X^{np}$, where $n$ is a fixed integer number, we derive
\begin{align}
\textup{d}\mathcal{K}(X)&\leq \mathcal{L}\mathcal{K} \textup{d}t+npX^{np-1}\Big(\sigma_1 S\textup{d}\mathcal{W}_1(t)+\sigma_2 I\textup{d}\mathcal{W}_2(t)+\sigma_3 Q\textup{d}\mathcal{W}_3(t)\Big)\nonumber\\&\;\;\;+\int_Z X^{np}(t^{-})\Big((1+\bar{\eta}(u))^{np}-1\Big)\widetilde{\mathcal{N}}(\textup{d}t,\textup{d}u),
\label{est9}
\end{align}
where
\begin{align*}
\mathcal{L}\mathcal{K}&\leq npX^{np-2}\Big[AX-\Big(\mu_1-\frac{(np-1)}{2}\bar{\sigma}-\frac{1}{np}\ell_{n,p}\Big)X^{2}\Big].
\end{align*}
Choose a positive constant $p>1$ such that $\Gamma_{n,p}=\mu_1-\frac{(np-1)}{2}\bar{\sigma}-\frac{1}{np}\ell_{n,p}>0$. Then
\begin{align}
\textup{d}\mathcal{K}(X)&\leq\Big(npX^{np-2}\big(AX-\Gamma_{n,p} X^2\big)\Big)\textup{d}t+npX^{np-1}\Big(\sigma_1 S\textup{d}\mathcal{W}_1(t)+\sigma_2 I\textup{d}\mathcal{W}_2(t)+\sigma_3 Q\textup{d}\mathcal{W}_3(t)\Big)\nonumber\\&\;\;\;+\int_Z X^{np}(t^{-})\Big((1+\bar{\eta}(u))^{np}-1\Big)\widetilde{\mathcal{N}}(\textup{d}t,\textup{d}u).
\label{levy1}
\end{align}
For any constant $m$ satisfying $m\in(0,np\Gamma_{n,p})$, one can see that
\begin{align*}
\textup{d}e^{m s}\mathcal{K}(X(s))&\leq\mathcal{L}\big(e^{mt}\mathcal{K}(X(t))\big)+npe^{mt}X^{np-1}(t)\Big(\sigma_1 S(t)\textup{d}\mathcal{W}_1(t)+\sigma_2 I(t)\textup{d}\mathcal{W}_2(t)+\sigma_3 Q(t)\textup{d}\mathcal{W}_3(t)\Big)\\&\;\;\;+e^{mt}\int_Z X^{np}(t^{-})\Big((1+\bar{\eta}(u))^{np}-1\Big)\widetilde{\mathcal{N}}(\textup{d}t,\textup{d}u).
\end{align*}
Integrating both sides of the last inequality from $0$ to $t$, we get
\begin{align*}
\int^t_0\textup{d}e^{m s}\mathcal{K}(X(s))&\leq \int^t_0 \Big(m e^{m s}\mathcal{K}(X(s))+e^{ms}\mathcal{L}\big(\mathcal{K}(X(s)\big)\Big)\textup{d}s\\&\;\;\;+np\int^t_0 e^{ms}X^{np-1}(s)\Big(\sigma_1 S(s)\textup{d}\mathcal{W}_1(s)+\sigma_2 I(s)\textup{d}\mathcal{W}_2(s)+\sigma_3 Q(s)\textup{d}\mathcal{W}_3(s)\Big)\\&\;\;\;+\int^t_0e^{ms}\int_Z X^{np}(s^{-})\Big((1+\bar{\eta}(u))^{np}-1\Big)\widetilde{\mathcal{N}}(\textup{d}s,\textup{d}u).
\end{align*}
Taking expectation on both sides yields that
\begin{align*}
\E e^{m t}\mathcal{K}(X(t))&\leq \mathcal{K}(X(0))+\E\Big\{\int^t_0 \Big(m e^{m s}\mathcal{K}(X(s))+e^{ms}\mathcal{L}\big(\mathcal{K}(X(s))\Big)\textup{d}s\Big\}.
\end{align*}
In view of (\ref{levy1}), we can see that
\begin{align*}
m e^{m t}\mathcal{K}(X(t))+e^{mt}\mathcal{L}\big(\mathcal{K}(X)\big)\leq np e^{m t} \bar{H},
\end{align*}
where $\bar{H}=\underset{X>0}{\sup}\Big\{X^{np-2}\Big[-\Big(\Gamma_{n,p}-\frac{m}{np}\Big)X^2+AX\Big]+1\Big\}$. Then, we have
\begin{align*}
\E e^{m t}\mathcal{K}(X(s))&\leq \mathcal{K}(X(0))+\frac{np\bar{H}}{m} e^{m t}.
\end{align*}
Therefore, we get 
\begin{align*}
\underset{t\to \infty}{\lim \sup}\; \E \big[X^{np}(t)\big]\leq \frac{np\bar{H}}{m}\;\; \mbox{a.s.}
\end{align*}
Consequently, there exists a positive constant $\bar{M}$ such that for all $t\geq 0$, 
\begin{align}
\E\big[X^{np}(t)\big]\leq \bar{M}.
\label{clef}
\end{align}

\textbf{Step 2.}\;\;Integrating from $0$ to $t$ after applying the famous Burkholder-Davis-Gundy inequality \cite{mo} to \eqref{levy1}, allows us to conclude that for an arbitrarily small positive constant $z$, $m=1,2,...$, 
\begin{align*}
\E\Big[\underset{m z\leq t\leq(m+1)z}{\sup}\;X^{np}(t)\Big]&\leq \E\Big[X(mz)\Big]^{np}+\Big(z_1z+z_2 z^\frac{1}{2}\big(np \bar{\sigma}+\int_Z\big((1+\bar{\eta}(u))^{np}-1\big)^2\nu(\textup{d}u)\big)\Big)\\&\;\;\;\times \Big[\underset{m z\leq t\leq(m+1)z}{\sup}\;X^{np}(t)\Big],
\end{align*}
where $z_1$ and $z_2$ are positive constants. Specially, we select $z>0$ such that
\begin{align*}
z_1z+z_2 z^\frac{1}{2}\Big(np \bar{\sigma}+\int_Z\big((1+\bar{\eta}(u))^{np}-1\big)^2\nu(\textup{d}u)\Big)\leq \frac{1}{2}.
\end{align*}
Then
\begin{align*}
\E\big[\underset{m z\leq t\leq(m+1)z}{\sup}\;X^{np}(t)\big]\leq 2\bar{M}.
\end{align*}
Let $\bar{\epsilon}>0$ be arbitrary. By employing Chebyshev's inequality, we derive
\begin{align*}
\P\Big\{\underset{m z\leq t\leq(m+1)z}{\sup} X^{np}(t)>(mz)^{1+\bar{\epsilon}}\Big\}\leq \frac{\E\Big[\underset{m z\leq t\leq(m+1)z}{\sup}\;X^{np}(t)\Big]}{(mz)^{1+\bar{\epsilon}}}\leq \frac{2\bar{M}}{(mz)^{1+\bar{\epsilon}}}.
\end{align*}
Making use of the Borel-Cantelli lemma gives that for almost all $\omega\in\Omega$
\begin{align}
\underset{m z\leq t\leq(m+1)z}{\sup} X^{np}(t)\leq(mz)^{1+\bar{\epsilon}},
\label{levy2}
\end{align}
verifies for all but finitely many $m$. Consequently, there exists a positive constant $m_0(\omega)$ such that $m_0\leq m$ and \eqref{levy2} holds for almost all $\omega\in\Omega$. In other words, for almost all $\omega\in\Omega$, if $m_0\leq m$ and $m z\leq t\leq(m+1)z$,
\begin{align*}
\frac{\ln X^{np}(t)}{\ln t}\leq \frac{(1+\bar{\epsilon})\ln (mz)}{\ln (mz)}=1+\bar{\epsilon}.
\end{align*}
Because $\bar{\epsilon}$ is arbitrarily small, then
\begin{align*}
\underset{t\to \infty}{\lim \sup}\frac{\ln X^{n}(t)}{\ln t}\leq \frac{1}{p}\;\;\; \mbox{a.s.}
\end{align*}
Therefore, for any small $\bar{v}\in(0,1-1/p)$, there is a constant $\bar{V}=\bar{V}(\omega)$, for which if $t\geq\bar{V}$ then
\begin{align*}
\ln X^{n}(t)\leq \Big(\frac{1}{p}+\bar{v}\Big)\ln t.
\end{align*}
Hence
\begin{align*}
\underset{t\to \infty}{\lim \sup}\frac{ X^{n}(t)}{ t}\leq \underset{t\to \infty}{\lim \sup}\frac{t^{\frac{1}{p}+\bar{v}}}{t}=0.
\end{align*}
This together with the positivity of the solution implies
\begin{align*}
\underset{t\to \infty}{\lim }\frac{ X^{n}(t)}{ t}=0\;\;\; \mbox{a.s.}
\end{align*}

\textbf{Step 3.}\;\;Now, we define
\begin{align*}
\mathcal{I}_1(t)&=\frac{1}{t}\int^t_0X(s)S(s)\textup{d}\mathcal{W}_1(s),\hspace{0.5cm}\mathcal{I}_2(t)=\frac{1}{t}\int^t_0X(s)I(s)\textup{d}\mathcal{W}_2(s),\\
\mathcal{I}_3(t)&=\frac{1}{t}\int^t_0X(s)Q(s)\textup{d}\mathcal{W}_3(s),\hspace{0.5cm} \mathcal{I}_4(t)=\frac{1}{ t}\int^t_0\int_ZX^2(s^{-})\big(\big(1+\bar{\eta}\big)^2-1\big)\widetilde{\mathcal{N}}(\textup{d}s,\textup{d}u).
\end{align*}
In view of the Burkholder–Davis–Gundy inequality, we find that for $\bar{p}>2$,
\begin{align}
\E\Bigg[\underset{m \leq t\leq(m+1)}{\sup}|\mathcal{I}_1(t)|^{\bar{p}}\Bigg]\leq C_{\bar{p}}\E\Bigg[\int^t_0X^4(s)\textup{d}s\Bigg]^{\frac{\bar{p}}{2}}\leq C_{\bar{p}}\Bigg[\E\int^t_0X^4(s)\textup{d}s\Bigg]^{\frac{\bar{p}}{2}}\leq C_{\bar{p}}\Bigg[\E\int^t_0|X^4(s)|\textup{d}s\Bigg]^{\frac{\bar{p}}{2}},
\label{clef2}
\end{align}
where $C_{\bar{p}}=\left[\frac{\bar{p}^{\bar{p}+1}}{2(\bar{p}-1)^{\bar{p}-1}}\right]^{\bar{p}/2}>0$. Similarly to the previous case, we find
\begin{align*}
\E\Bigg[\underset{m \leq t\leq(m+1)}{\sup}|\mathcal{I}_4(t)|^{\bar{p}}\Bigg]&\leq C_{\bar{p}}\Bigg(\int_Z \Big(\big(1+\bar{\eta}\big)^{2}-1\Big)^2\nu(\textup{d}u)\Bigg)^{\frac{{\bar{p}}}{2}}\Bigg[\E\int^t_0|X^4(s)|\textup{d}s\Bigg]^{\frac{{\bar{p}}}{2}}.
\end{align*}
Via \eqref{clef} and \eqref{clef2}, one can see that
\begin{align*}
\E\Big[\underset{m \leq t\leq(m+1)}{\sup}|\mathcal{I}_1(t)|^{{\bar{p}}}\Big]&\leq  2^{1+\frac{\bar{p}}{2}}\bar{M}C_{\bar{p}}m^\frac{\bar{p}}{2}.
\end{align*}
For any arbitrary positive constant $\tilde{\epsilon}$, and by making use of Chebyshev's inequality, we obtain
\begin{align*}
\P\Big\{\underset{m \leq t\leq(m+1)}{\sup}|\mathcal{I}_1(t)|^{{\bar{p}}}>\bar{p}^{1+\tilde{\epsilon}+\frac{\bar{p}}{2}}\Big\}\leq \frac{\E\Big[\underset{m \leq t\leq(m+1)}{\sup}|\mathcal{I}_1(m+1)|^{{\bar{p}}}\Big]}{\bar{p}^{1+\tilde{\epsilon}+\frac{\bar{p}}{2}}}\leq \frac{2^{1+\frac{\bar{p}}{2}}\bar{M}C_{\bar{p}}}{\bar{p}^{1+\tilde{\epsilon}}},\hspace*{0.3cm}m=1,2,...
\end{align*}
Using the Borel-Cantelli lemma, one has
\begin{align*}
\frac{\ln|\mathcal{I}_1(t)|^{\bar{p}}}{\ln t}\leq \frac{\big(1+\tilde{\epsilon}+\frac{\bar{p}}{2}\big)\ln m}{\ln m}=1+\tilde{\epsilon}+\frac{\bar{p}}{2}.
\end{align*}
Taking the limit superior on both sides of the last inequality and applying the arbitrariness of $\tilde{\epsilon}$, we deduce
\begin{align*}
\underset{t\to \infty}{\lim \sup}\frac{\ln|\mathcal{I}_1(t)|}{\ln t}\leq \frac{1}{2}+\frac{1}{\bar{p}}\;\;\; \mbox{a.s.}
\end{align*}
That is to say, for any positive constant $\bar{\tau}\in\big(0,\frac{1}{2}-\frac{1}{\bar{p}}\big)$, there exists a constant $\bar{T}=\bar{T}(\omega)$ such that for all $t\geq \bar{T}$,
\begin{align*}
\ln|\mathcal{I}_1(t)|\leq\Big(\frac{1}{2}+\frac{1}{\bar{p}}+\bar{\tau}\Big)\ln t.
\end{align*}
Dividing both sides of the last inequality by $t$ and taking the limit superior, we have
\begin{align*}
\underset{t\to \infty}{\lim \sup}\frac{|\mathcal{I}_1(t)|}{t}\leq \underset{t\to \infty}{\lim \sup}\frac{t^{\frac{1}{2}+\frac{1}{\bar{p}}+\bar{\tau}}}{t}=0.
\end{align*}
Combining it with $\underset{t\to \infty}{\lim \inf}\frac{|\mathcal{I}_1(t)|}{t}\geq 0$, one has $\underset{t\to \infty}{\lim }\frac{|\mathcal{I}_1(t)|}{t}=\underset{t\to \infty}{\lim }\frac{\mathcal{I}_1(t)}{t}=0\;\;\; \mbox{a.s.}$\\
In the same way, we prove that
\begin{align*}
\underset{t\to \infty}{\lim }\frac{\mathcal{I}_2(t)}{t}=0,\hspace{0.3cm}\underset{t\to \infty}{\lim }\frac{\mathcal{I}_3(t)}{t}=0,\hspace{0.3cm}\underset{t\to \infty}{\lim }\frac{\mathcal{I}_4(t)}{t}=0\;\;\; \mbox{a.s.}
\end{align*}
This completes the proof.
\end{proof}
\begin{rem}
The positivity of the solutions $X(t)$ and $Y(t)$ together with the stochastic comparison theorem \cite{mo}, leads to the fact that $N(t)\leq X(t)$ a.s. which in turn implies that 
\begin{align*}
\underset{t\to\infty}{\lim}\frac{S^n(t)}{t}=0, \hspace{0.3cm}  \underset{t\to\infty}{\lim}\frac{I^n(t)}{t}=0, \hspace{0.3cm} \underset{t\to\infty}{\lim}\frac{Q^n(t)}{t}=0,  ~~~\mbox{and even}~~~ \underset{t\to\infty}{\lim} \frac{N^n(t)}{t}=0 \hspace{0.3cm}\mbox{a.s.}
\end{align*}
\label{comp}
\end{rem}
\begin{rem}
By comparing our findings with those of Lemmas 3.3 and 3.4 in \cite{166}, one can conclude that the new result \ref{lem22} presents a modified and generalized version to these lemmas, which will be necessary to prove Lemma \ref{lemmas}.
\end{rem}
\begin{lem}
Let $Y(0)\in\R^3_+$ be a positive given value. If $Y(t)$ denotes the positive solution of system (\ref{s2}) that starts from $Y(0)$, then
\begin{align*}
&\bullet \underset{t\to\infty}{\lim}\frac{\int^t_0 S(s)\textup{d}\mathcal{W}_{\beta}(s)}{t}=0, \hspace{0.3cm}\underset{t\to\infty}{\lim}\frac{\int^t_0 S(s)\textup{d}\mathcal{W}_1(s)}{t}=0, \;\;\underset{t\to\infty}{\lim}\frac{\int^t_0 I(s)\textup{d}\mathcal{W}_2(s)}{t}=0,  \hspace{0.2cm}\mbox{and}\hspace{0.2cm} \underset{t\to\infty}{\lim}\frac{\int^t_0 Q(s)\textup{d}\mathcal{W}_3(s)}{t}=0 \hspace{0.3cm}\mbox{a.s.}\\
&\bullet \underset{t\to\infty}{\lim}\frac{\int^t_0\int_Z \eta_1(u)S(s^{-})\widetilde{\mathcal{N}}(\textup{d}s,\textup{d}u)}{t}=0,\;\underset{t\to\infty}{\lim}\frac{\int^t_0\int_Z \eta_2(u)I(s^{-})\widetilde{\mathcal{N}}(\textup{d}s,\textup{d}u)}{t}=0,\;\underset{t\to\infty}{\lim}\frac{\int^t_0\int_Z \eta_3(u)Q(s^{-})\widetilde{\mathcal{N}}(\textup{d}s,\textup{d}u)}{t}=0\hspace{0.3cm}\mbox{a.s.}
\end{align*}
\label{lem1m}
\end{lem}
\begin{rem}
The last lemma is easily demonstrated by using an analysis similar to that in the proof of Lemma \ref{lem22}.
\end{rem} 
\begin{rem}
In the absence of Lévy noise (see for example \cite{sta}), the stationary distribution expression is used to calculate the time averages of the auxiliary process solution by employing the ergodic theorem \cite{mo}. Unfortunately, the said expression is still unknown in the case of the Lévy jumps. This problem is implicitly mentioned in \cite{58} as an open question, and the authors presented the threshold analysis of their model with an unknown stationary distribution formula. In this article, we propose an alternative method to establish the exact expression of the threshold parameter without having recourse to the use of ergodic theorem. This new idea that we propose is presented in the following lemma.
\end{rem}
\begin{lem}
Let $X(t)\in\R_+$ be the solution of the equation \eqref{s4} with any given initial value $X(0)=N(0)\in\R_+$. Suppose that $\chi=2\mu_1-\bar{\sigma}-\int_Z\big[\bar{\eta}^2(u)\vee\underline{\eta}^2(u)\big]\nu(\textup{d}u)>0$, then 
\begin{align*}
\underset{t\to\infty}{\lim}\frac{1}{t}\int^t_0X(s)\textup{d}s=\frac{A}{\mu_1}\hspace{0.2cm}\mbox{a.s.}
\end{align*} 
and
\begin{align*}
\underset{t\to\infty}{\lim}\frac{1}{t}\int^t_0X^2(s)\textup{d}s\leq\frac{2A^2}{\mu_1\chi}\hspace{0.2cm}\mbox{a.s.}
\end{align*}
\label{lemmas}
\end{lem}
\begin{proof}
Integrating from $0$ to $t$ on both sides of (\ref{s4}) yields
\begin{align*}
\frac{X(t)-X(0)}{t}&=A-\frac{\mu_1}{t}\int^t_0X(s)\textup{d}s+\frac{\sigma_1}{t}\int_0^t S(s)\textup{d}\mathcal{W}_1(s)+\frac{1}{t}\int^t_0\int_Z \eta_1(u)S(s^{-})\widetilde{\mathcal{N}}(\textup{d}s,\textup{d}u)\\&\;\;\;+\frac{\sigma_2}{t}\int_0^t I(s)\textup{d}\mathcal{W}_2(s)+\frac{1}{t}\int^t_0\int_Z \eta_2(u)I(s^{-})\widetilde{\mathcal{N}}(\textup{d}s,\textup{d}u)\\&\;\;\;+\frac{\sigma_3}{t}\int_0^t Q(s)\textup{d}\mathcal{W}_3(s)+\frac{1}{t}\int^t_0\int_Z \eta_3(u)Q(s^{-})\widetilde{\mathcal{N}}(\textup{d}s,\textup{d}u).
\end{align*}
Clearly, we can derive that
\begin{align*}
\frac{1}{t}\int^t_0X(s)\textup{d}s&=\frac{A}{\mu_1}-\frac{X(t)-X(0)}{\mu_1 t}+\frac{\sigma_1}{\mu_1 t}\int_0^t S(s)\textup{d}\mathcal{W}_1(s)+\frac{1}{\mu_1 t}\int^t_0\int_Z \eta_1(u)S(s^{-})\widetilde{\mathcal{N}}(\textup{d}s,\textup{d}u)\\&\;\;\;+\frac{\sigma_2}{\mu_1 t}\int_0^t I(s)\textup{d}\mathcal{W}_2(s)+\frac{1}{\mu_1 t}\int^t_0\int_Z \eta_2(u)I(s^{-})\widetilde{\mathcal{N}}(\textup{d}s,\textup{d}u)\\&\;\;\;+\frac{\sigma_3}{\mu_1 t}\int_0^t Q(s)\textup{d}\mathcal{W}_3(s)+\frac{1}{\mu_1 t}\int^t_0\int_Z \eta_3(u)Q(s^{-})\widetilde{\mathcal{N}}(\textup{d}s,\textup{d}u).
\end{align*}
According to Lemma \ref{lem1m}, we have
\begin{align*}
\underset{t\to\infty}{\lim}\frac{1}{t}\int^t_0X(s)\textup{d}s=\frac{A}{\mu_1}\hspace{0.2cm}\mbox{a.s.}
\end{align*}
Now, applying the generalized Itô’s formula to equation (\ref{s4}) leads to
\begin{align}
\text{d}X^2(t)&\leq\bigg(2X(t)\Big(A-\mu_1X(t)\Big)+\bar{\sigma}X^2(t)+\int_ZX^2(t)\big[\bar{\eta}^2(u)\vee\underline{\eta}^2(u)\big]\nu(\textup{d}u)\bigg)\textup{d}t\nonumber\\&\;\;\;+2X(t)\Big(\sigma_1 S(t)\textup{d}\mathcal{W}_1(t)+\sigma_2 I(t)\textup{d}\mathcal{W}_2(t)+\sigma_3 Q(t)\textup{d}\mathcal{W}_3(t)\Big)\nonumber\\&\;\;\;+\int_Z X^2(t^{-})\Big(\big(1+\bar{\eta}(u)\big)^2-1\Big)\widetilde{\mathcal{N}}(\textup{d}t,\textup{d}u).
\label{maj2}
\end{align}
Integrating both sides from $0$ to $t$, yields
\begin{align*}
X^2(t)-X^2(0)&\leq 2A\int^t_0X(s)\textup{d}s-\bigg(2\mu_1-\bar{\sigma}-\int_Z\big[\bar{\eta}^2(u)\vee\underline{\eta}^2(u)\big]\nu(\textup{d}u)\bigg)\int^t_0X^2(s)\textup{d}s\\&\;\;\;+2\sigma_1\int^t_0X(s)S(s)\textup{d}\mathcal{W}_1(s)+2\sigma_2\int^t_0X(s)I(s)\textup{d}\mathcal{W}_2(s)+2\sigma_3\int^t_0X(s)Q(s)\textup{d}\mathcal{W}_3(s)\\&\;\;\;+\int^t_0\int_ZX^2(s^{-})\Big(\big(1+\bar{\eta}(u)\big)^2-1\Big)\widetilde{\mathcal{N}}(\textup{d}s,\textup{d}u).
\end{align*}
Let $\chi=2\mu_1-\bar{\sigma}-\int_Z\big[\bar{\eta}^2(u)\vee\underline{\eta}^2(u)\big]\nu(\textup{d}u)>0$. Therefore
\begin{align*}
\frac{1}{t}\int^t_0X^2(s)\textup{d}s&\leq \frac{2A}{\chi t}\int^t_0X(s)\textup{d}s
+\frac{X^2(0)-X^2(t)}{\chi t}+\frac{2\sigma_1}{\chi t}\int^t_0X(s)S(s)\textup{d}\mathcal{W}_1(s)\\&\;\;\;+\frac{2\sigma_2}{\chi t}\int^t_0X(s)I(s)\textup{d}\mathcal{W}_2(s)+\frac{2\sigma_3}{\chi t}\int^t_0X(s)Q(s)\textup{d}\mathcal{W}_3(s)\\&\;\;\;+\frac{1}{\chi t}\int^t_0\int_ZX^2(s)\Big(\big(1+\bar{\eta}(u)\big)^2-1\Big)\widetilde{\mathcal{N}}(\textup{d}s,\textup{d}u).
\end{align*}
By Lemma \ref{lem22} and assumptions \textbf{($\textup{A}_4$)}-\textbf{($\textup{A}_5$)}, we can easily verify that
\begin{align*}
\underset{t\to\infty}{\lim}\frac{1}{t}\int^t_0X^2(s)\textup{d}s\leq\frac{2A^2}{\mu_1\chi}\hspace{0.2cm}\mbox{a.s.}
\end{align*}
\end{proof}
Now, we present a lemma which gives mutually exclusive possibilities for the existence of an ergodic stationary distribution to the system (\ref{s2}).
\begin{lem}[\textit{Mutually exclusive possibilities lemma}, \cite{10}]
Let $\phi(t)\in \R^n$ be a stochastic Feller process, then either an ergodic probability measure exists, or 
\begin{align}
\underset{t\to \infty}{\lim }\underset{\hat{\nu}}{\sup }\frac{1}{t}\int^t_0\int \P(s,x,\Sigma)\hat{\nu}(dx)\textup{d}s=0,
\label{imp}
\end{align}
for any compact set $\Sigma\subset\R^n$, where the supremum is taken over all initial distributions $\hat{\nu}$ on $\R^n$ and $\P(t,x,\Sigma)$ is the probability for $\phi(t)\in \Sigma$ with $\phi(0)=x\in \R^n$.
\label{poss}
\end{lem}
\section{Long-term dynamics of the stochastic system (\ref{s2})} \label{sec3}
\subsection{Ergodicity and persistence in the mean}
In the following, we aim to give the condition for the ergodicity the persistence of the disease. We suppose that $\chi>0$ and we define the parameter:
\begin{align*}
\mathcal{R}^s_0&=\Big(\mu_2+\delta+\gamma+\frac{\sigma_2^2}{2}\Big)^{-1}\left(\frac{\beta A}{\mu_1}-\frac{A^2\sigma_{\beta}^2}{\mu_1\chi}-\int_Z \eta_2(u)-\ln(1+\eta_2(u))\nu(\textup{d}u)\right).
\end{align*}
For simplicity, we introduce the following notations:
\begin{align*}
M_1&=\frac{\mu_1^2}{4\beta^2A}\Big(\mu_2+\delta+\gamma+\frac{\sigma_2^2}{2}\Big)\Big(\mathcal{R}_0^s-1\Big),\\
M_2&=\frac{p\mu_1 \Gamma_{2,p}\beta^{-(p+1)}}{8 \Delta}\Big(\mu_2+\delta+\gamma+\frac{\sigma_2^2}{2}\Big)\Big(\mathcal{R}_0^s-1\Big),\\
M_3&=\frac{\mu_1 q}{8\beta}\Big(\mu_2+\delta+\gamma+\frac{\sigma_2^2}{2}\Big)\Big(\frac{2A}{\mu_1}+N(0)\Big)^{-1}\Big(\mathcal{R}_0^s-1\Big).
\end{align*}
\begin{thm}
If $\mathcal{R}^s_0>1$, the stochastic system (\ref{s2}) admits a unique stationary distribution and it has the ergodic property for any initial value $Y(0)\in\R^3_+$.
\label{thm1}
\end{thm}
\begin{proof}
Motivated by the proof of Lemma 3.2 in \cite{11}, we briefly verify the Feller property of the stochastic model (\ref{s2}). The
main purpose of the next step is to prove that (\ref{imp}) is impossible. Applying the generalized Itô’s formula to $\ln I-\frac{\beta}{\mu_1}(X-S)$, we easily derive
\begin{align*}
\text{d}\Bigg\{\ln I(t)-\frac{\beta}{\mu_1}\Big(X(t)-S(t)\Big)\Bigg\}&=\bigg(\beta S(t)-(\mu_2+\delta+\gamma)-\frac{\sigma_2^2}{2}-\frac{\sigma_{\beta}^2}{2}S^2(t)-\int_Z \eta_2(u)-\ln(1+\eta_2(u))\nu(\textup{d}u)\bigg)\textup{d}t\\&\;\;\;-\frac{\beta}{\mu_1}\Big(-\mu_1(X(t)-S(t))+\beta S(t)I(t)-\gamma I(t)-kQ(t)\Big)\textup{d}t+\sigma_2\textup{d}\mathcal{W}_2(t)\\&\;\;\;+\int_Z \ln(1+\eta_2(u))\widetilde{\mathcal{N}}(\textup{d}t,\textup{d}u)+\sigma_{\beta} S(t)\textup{d}\mathcal{W}_{\beta}(t)-\frac{\beta}{\mu_1}\sigma_{\beta} S(t)I(t)\textup{d}\mathcal{W}_{\beta}(t)\nonumber\\&\;\;\;-\frac{\beta}{\mu_1}\bar{\mathcal{P}}_2(t)-\frac{\beta}{\mu_1}\mathcal{P}_3(t).
\end{align*}
Then
\begin{align}
\text{d}\Bigg\{\ln I(t)-\frac{\beta}{\mu_1}\Big(X(t)-S(t)\Big)\Bigg\}&\geq\bigg(\beta X(t)-(\mu_2+\delta+\gamma)-\frac{\sigma_2^2}{2}-\frac{\sigma_{\beta}^2}{2}S^2(t)-\int_Z \eta_2(u)-\ln(1+\eta_2(u))\nu(\textup{d}u)\bigg)\textup{d}t\nonumber\\&\;\;\;-\frac{\beta^2 }{\mu_1} S(t)I(t)\textup{d}t+\sigma_2\textup{d}\mathcal{W}_2(t)+\int_Z \ln(1+\eta_2(u))\widetilde{\mathcal{N}}(\textup{d}t,\textup{d}u)+\sigma_{\beta} S(t)\textup{d}\mathcal{W}_{\beta}(t)\nonumber\\&\;\;\;-\frac{\beta}{\mu_1}\sigma_{\beta} S(t)I(t)\textup{d}\mathcal{W}_{\beta}(t)-\frac{\beta}{\mu_1}\bar{\mathcal{P}}_2(t)-\frac{\beta}{\mu_1}\mathcal{P}_3(t).
\label{ito}
\end{align}
Integrating from $0$ to $t$ on both sides of (\ref{ito}) yields
\begin{align*}
&\ln\frac{I(t)}{I(0)}-\frac{\beta}{\mu_1}\Big(X(t)-S(t)\Big)+\frac{\beta}{\mu_1}\Big(X(0)-S(0)\Big)\\&\geq\int^t_0\bigg(\beta X(s)-(\mu_2+\delta+\gamma)-\frac{\sigma_2^2}{2}-\frac{\sigma_{\beta}^2}{2}S^2(s)-\int_Z \eta_2(u)-\ln(1+\eta_2(u))\nu(\textup{d}u)\bigg)\textup{d}s\nonumber\\&\;\;\;-\frac{\beta^2 }{\mu_1}\int^t_0 S(s)I(s)\textup{d}s+\sigma_2W_2(t)+\int^t_0\int_Z \ln(1+\eta_2(u))\widetilde{\mathcal{N}}(\textup{d}s,\textup{d}u)+\sigma_{\beta} \int^t_0S(s)\textup{d}\mathcal{W}_{\beta}(s)\nonumber\\&\;\;\;-\frac{\beta\sigma_{\beta}}{\mu_1}\int^t_0 S(s)I(s)\textup{d}\mathcal{W}_{\beta}(s)-\frac{\beta\sigma_2}{\mu_1}\int^t_0 I(s) \textup{d}\mathcal{W}_2(s)-\frac{\beta}{\mu_1}\int^t_0\int_Z \eta_2(u)I(s^{-})\widetilde{\mathcal{N}}(\textup{d}s,\textup{d}u)\nonumber\\&\;\;\;-\frac{\beta\sigma_3}{\mu_1}\int^t_0 Q(s) \textup{d}\mathcal{W}_3(s)-\frac{\beta}{\mu_1}\int^t_0\int_Z \eta_3(u)Q(s^{-})\widetilde{\mathcal{N}}(\textup{d}s,\textup{d}u).
\end{align*}
Hence
\begin{align}
&\int^t_0\beta S(s)I(s)\textup{d}s\nonumber\\&\geq \frac{\mu_1}{\beta}\int^t_0\bigg(\beta X(s)-(\mu_2+\delta+\gamma)-\frac{\sigma_2^2}{2}-\frac{\sigma_{\beta}^2}{2}S^2(s)-\int_Z \eta_2(u)-\ln(1+\eta_2(u))\nu(\textup{d}u)\bigg)\textup{d}s\nonumber\\&\;\;\;-\frac{\mu_1}{\beta}\ln\frac{I(t)}{I(0)}+\Big(X(t)-S(t)\Big)-\Big(X(0)-S(0)\Big)\nonumber+\frac{\mu_1}{\beta}\sigma_2W_2(t)+\frac{\mu_1}{\beta}\int^t_0\int_Z \ln(1+\eta_2(u))\widetilde{\mathcal{N}}(\textup{d}s,\textup{d}u)\nonumber\\&\;\;\;+\frac{\mu_1\sigma_{\beta} }{\beta}\int^t_0S(s)\textup{d}\mathcal{W}_{\beta}(s)-\sigma_{\beta}\int^t_0 S(s)I(s)\textup{d}\mathcal{W}_{\beta}(s)-\sigma_2\int^t_0 I(s) \textup{d}\mathcal{W}_2(s)-\int^t_0\int_Z \eta_2(u)I(s^{-})\widetilde{\mathcal{N}}(\textup{d}s,\textup{d}u)\nonumber\\&\;\;\;-\beta\sigma_3\int^t_0 Q(s) \textup{d}\mathcal{W}_3(s)-\int^t_0\int_Z \eta_3(u)Q(s^{-})\widetilde{\mathcal{N}}(\textup{d}s,\textup{d}u).
\label{11}
\end{align}
From Remark \ref{comp} and Lemma \ref{lem1m}, one can derive that 
\begin{align*}
\underset{t\to\infty}{\lim}\frac{X(t)}{t}=0,\hspace{0.2cm}  \underset{t\to\infty}{\lim}\frac{S(t)}{t}=0,\hspace{0.2cm}\mbox{and}\hspace{0.2cm}
\underset{t\to\infty}{\lim}\frac{1}{t}\int^t_0\int_Z (\eta_2(u) I(s^{-})+ \eta_3(u)Q(s^{-}))\widetilde{\mathcal{N}}(\textup{d}s,\textup{d}u)\hspace{0.5cm}\mbox{a.s.}
\end{align*}
Moreover, 
\begin{align*}
\underset{t\to\infty}{\lim}\frac{1}{ t} \int^t_0S(s)\textup{d}\mathcal{W}_{\beta}(s)=0,\;\;\;  \frac{1}{t}\int^t_0I(s)\textup{d}\mathcal{W}_2(s)=0\hspace{0.2cm}\mbox{and}\hspace{0.2cm} \underset{t\to\infty}{\lim} \frac{1}{t}\int^t_0Q(s)\textup{d}\mathcal{W}_3(s)=0\hspace{0.5cm}\mbox{a.s.}
\end{align*}
Application of the strong law of large numbers and assumption \textbf{($\textup{A}_3$)} shows that
$$\underset{t\to \infty}{\lim} \frac{W_2(t)}{ t}=0\hspace{0.2cm}\mbox{and}\hspace{0.2cm}\underset{t\to \infty}{\lim}\frac{1}{ t}\int^t_0\int_Z\ln(1+\eta_2(u))\tilde{\mathcal{N}}(\textup{d}s,\textup{d}u)=0\hspace{0.5cm}\mbox{a.s.}$$
Applying similar arguments to those in the proof of Lemma \ref{lem22}, we obtain
\begin{align*}
\underset{t\to \infty}{\lim}\frac{1}{t}\int^t_0S(s)I(s)\textup{d}\mathcal{W}_{\beta}(s)=0\hspace{0.5cm}\mbox{a.s.}
\end{align*}
Since $\underset{t\to \infty}{\lim \sup}\frac{1}{t}\ln\frac{I(t)}{I(0)}\leq \underset{t\to \infty}{\lim \sup}\frac{1}{t}\ln\frac{N(t)}{I(0)}\leq 0$ a.s., one can derive that
\begin{align*}
&\underset{t\to \infty}{\lim \inf} \frac{1}{t}\int^t_0 \beta S(s)I(s)\textup{d}s\nonumber\\&\geq \frac{\mu_1}{\beta}\underset{t\to \infty}{\lim \inf} \frac{1}{t}\int^t_0\bigg(\beta X(s)-(\mu_2+\delta+\gamma)-\frac{\sigma_2^2}{2}-\frac{\sigma_{\beta}^2}{2}X^2(s)-\int_Z \eta_2(u)-\ln(1+\eta_2(u))\nu(\textup{d}u)\bigg)\textup{d}s\nonumber\\
&=\frac{\mu_1}{\beta}\underset{t\to \infty}{\lim } \frac{1}{t}\int^t_0\beta X(s)\textup{d}s-\frac{\sigma_{\beta}^2}{2}\underset{t\to \infty}{\lim } \frac{1}{t}\int^t_0X^2(s)\textup{d}s-\Big(\mu_2+\delta+\gamma+\frac{\sigma_2^2}{2}\Big)-\int_Z \eta_2(u)-\ln(1+\eta_2(u))\nu(\textup{d}u).
\end{align*}
From Lemma \ref{lemmas}, it follows that
\begin{align}
\underset{t\to \infty}{\lim \inf} \frac{1}{t}\int^t_0 \beta S(s)I(s)\textup{d}s&\geq \frac{\mu_1}{\beta}\times\Bigg(\frac{\beta A}{\mu_1}-\frac{A^2\sigma_{\beta}^2}{\mu_1\chi}-\Big(\mu_2+\delta+\gamma+\frac{\sigma_2^2}{2}\Big)-\int_Z \eta_2(u)-\ln(1+\eta_2(u))\nu(\textup{d}u)\Bigg)\nonumber\\&=\frac{\mu_1}{\beta}\Big(\mu_2+\delta+\gamma+\frac{\sigma_2^2}{2}\Big)\Big(\mathcal{R}_0^s-1\Big)>0\hspace{0.5cm}\mbox{a.s.}
\label{ergo1}
\end{align}
To continue our analysis, we need to set the following subsets: 
\begin{align*}
\Omega_1&=\{(S,I,Q)\in\R^3_+|\hspace{0.1cm}S\geq \epsilon,\hspace{0.1cm}\mbox{and},\hspace{0.1cm} I\geq \epsilon\},\\
\Omega_2&=\{(S,I,Q)\in\R^3_+|\hspace{0.1cm}S\leq \epsilon\},\\
\Omega_3&=\{(S,I,Q)\in\R^3_+|\hspace{0.1cm}I\leq \epsilon\},
\end{align*}
where $\epsilon>0$ is a positive constant to be determined later. Therefore, by (\ref{ergo1}), we get
\begin{align*}
&\underset{t\to \infty}{\lim \inf}\frac{1}{t}\int^t_0\E\Big( \beta S(s)I(s)\mathbf{1}_{\Omega_1}\Big)\textup{d}s\\&\geq \underset{t\to \infty}{\lim \inf}\frac{1}{t}\int^t_0\E\Big( \beta S(s)I(s)\Big)\textup{d}s-\underset{t\to \infty}{\lim \sup}\frac{1}{t}\int^t_0\E\Big( \beta S(s)I(s)\mathbf{1}_{\Omega_2}\Big)\textup{d}s-\underset{t\to \infty}{\lim \sup}\frac{1}{t}\int^t_0\E\Big( \beta S(s)I(s)\mathbf{1}_{\Omega_3}\Big)\textup{d}s\\&\geq \frac{\mu_1}{\beta}\Big(\mu_2+\delta+\gamma+\frac{\sigma_2^2}{2}\Big)\Big(\mathcal{R}_0^s-1\Big)-\beta \epsilon \underset{t\to \infty}{\lim \sup}\frac{1}{t}\int^t_0\E\big[I(s)\big]\textup{d}s-\beta \epsilon \underset{t\to \infty}{\lim \sup}\frac{1}{t}\int^t_0\E\big[S(s)\big]\textup{d}s.
\end{align*}
Then, one can see that
\begin{align*}
\underset{t\to \infty}{\lim \inf}\frac{1}{t}\int^t_0\E\Big( \beta S(s)I(s)\mathbf{1}_{\Omega_1}\Big)\textup{d}s&\geq \frac{\mu_1}{\beta}\Big(\mu_2+\delta+\gamma+\frac{\sigma_2^2}{2}\Big)\Big(\mathcal{R}_0^s-1\Big)- \frac{2A\beta \epsilon}{\mu_1}.
\end{align*}
We can choose $\epsilon \leq M_1$, and then we obtain
\begin{align}
\underset{t\to \infty}{\lim \inf}\frac{1}{t}\int^t_0\E\Big( \beta S(s)I(s)\mathbf{1}_{\Omega_1}\Big)\textup{d}s\geq \frac{\mu_1}{2\beta}\Big(\mu_2+\delta+\gamma+\frac{\sigma_2^2}{2}\Big)\Big(\mathcal{R}_0^s-1\Big)>0.
\label{22}
\end{align}
Let $q=a_0>1$ be a positive integer and $1<p=\frac{a_0}{a_0-1}$ such that $\Gamma_{2,p}> 0$ and $\frac{1}{q}+\frac{1}{p}=1$. By utilizing the Young inequality $xy\leq \frac{x^p}{p}+\frac{y^q}{q}$ for all $x$,$y>0$, we get
\begin{align*}
\underset{t\to \infty}{\lim \inf}\frac{1}{t}\int^t_0\E\Big( \beta S(s)I(s)\mathbf{1}_{\Omega_1}\Big)\textup{d}s&\leq \underset{t\to \infty}{\lim \inf}\frac{1}{t}\int^t_0\E\bigg(p^{-1}(\eta \beta S(s)I(s))^p+q^{-1}\eta^{-q}\mathbf{1}_{\Omega_1}\bigg)\textup{d}s\\&\leq \underset{t\to \infty}{\lim \inf}\frac{1}{t}\int^t_0 \E\Big(q^{-1}\eta^{-q}\mathbf{1}_{\Omega_1}\Big)\textup{d}s+p^{-1}(\eta \beta)^p\underset{t\to \infty}{\lim \sup}\frac{1}{t}\int^t_0\E\Big[N^{2p}(s)\Big]\textup{d}s,
\end{align*}
where $\eta$ is a positive constant satisfying $\eta^p\leq M_2$. By Lemma \ref{L1} and (\ref{22}), we deduce that
\begin{align}
\underset{t\to \infty}{\lim \inf}\frac{1}{t}\int^t_0\E\big[\mathbf{1}_{\Omega_1}\big]\textup{d}s&\geq q\eta^q\Bigg(\frac{\mu_1}{2\beta}\Big(\mu_2+\delta+\gamma+\frac{\sigma_2^2}{2}\Big)\Big(\mathcal{R}_0^s-1\Big)-\frac{2\eta^p\beta^p\Gamma_{2,p}}{p\Delta}\Bigg)\nonumber\\&\geq \frac{\mu_1 q\eta^q}{4\beta}\Big(\mu_2+\delta+\gamma+\frac{\sigma_2^2}{2}\Big)\Big(\mathcal{R}_0^s-1\Big)>0.
\label{23}
\end{align}
Setting $\Omega_4=\{(S,I,Q)\in\R^3_+|\hspace{0.1cm}S\geq \zeta,\hspace{0.1cm}\mbox{or},\hspace{0.1cm} I\geq \zeta\}$ and $\Sigma=\{(S,I,Q)\in\R^3_+|\hspace{0.1cm}\epsilon \leq S\leq \zeta,\hspace{0.1cm}\mbox{and},\hspace{0.1cm} \epsilon \leq I\leq \zeta\}$ where $\zeta>\epsilon$ is a positive constant to be explained in the following. By using the Tchebychev inequality, we can observe that
\begin{align*}
\E[\mathbf{1}_{\Omega_4}]\leq \P(S(t)\geq \zeta )+\P(I(t)\geq \zeta )&\leq\frac{1}{\zeta}\E[S(t)+I(t)]\leq\frac{1}{\zeta}\bigg(\frac{2A}{\mu_1}+N(0)\bigg).
\end{align*}
Choosing $\frac{1}{\zeta}\leq M_3\eta^q.$ We thus obtain
\begin{align*} 
\underset{t\to \infty}{\lim \sup}\frac{1}{t}\int^t_0\E[\mathbf{1}_{\Omega_4}]\textup{d}s&\leq \frac{\mu_1 q\eta^q}{8\beta}\Big(\mu_2+\delta+\gamma+\frac{\sigma_2^2}{2}\Big)\Big(\mathcal{R}_0^s-1\Big).
\end{align*}
According to (\ref{23}), one can derive that
\begin{align*}
\underset{t\to \infty}{\lim \inf}\frac{1}{t}\int^t_0\E[\mathbf{1}_{\Sigma}]\textup{d}s&\geq \underset{t\to \infty}{\lim \inf}\frac{1}{t}\int^t_0\E[\mathbf{1}_{\Omega_1}]\textup{d}s-\underset{t\to \infty}{\lim \sup}\frac{1}{t}\int^t_0\E[\mathbf{1}_{\Omega_4}]\textup{d}s\\&\;\;\; \geq\frac{\mu_1 q\eta^q}{8\beta}\Big(\mu_2+\delta+\gamma+\frac{\sigma_2^2}{2}\Big)\Big(\mathcal{R}_0^s-1\Big)>0.
\end{align*}
Based on the above analysis, 
we have determined a compact domain $\Sigma\subset \R^3_+$ such that
\begin{align*}
\underset{t\to \infty}{\lim \inf}\frac{1}{t}\int^t_0\P\Big(s,Y(0),\Sigma\Big)\textup{d}s\geq \frac{\mu_1 q\eta^q}{8\beta}\Big(\mu_2+\delta+\gamma+\frac{\sigma_2^2}{2}\Big)\Big(\mathcal{R}_0^s-1\Big)>0.
\end{align*}
Applying similar arguments to those in Theorem 5.1 of \cite{12}, we show the uniqueness of the ergodic stationary distribution of our model (\ref{s2}). This completes the proof.
\end{proof}
\begin{thm}
If $\mathcal{R}^s_0>1$, then for any value $Y(0)\in\R^3_+$, the disease is persistent in the mean. That is to say 
\begin{align*}
\underset{t\to\infty}{\lim\inf}\frac{1}{t}\int^t_0 I(s)\textup{d}s>0\hspace{0.2cm}\mbox{a.s.}
\end{align*}
\label{persis}
\end{thm}
\begin{proof}
From model (\ref{s2}) it yields
\begin{align}
\text{d}(S(t)+I(t)+Q(t))&=\big(A-\mu_1S(t)-\mu_2 I(t)-\mu_3Q(t)\big)\textup{d}t+\bar{\mathcal{P}}_1(t)+\bar{\mathcal{P}}_2(t)+\mathcal{P}_3(t).
\label{bymod}
\end{align}
Integrating (\ref{bymod}) from $0$ to $t$, and then dividing $t$ on both sides, we get
\begin{align*}
&\frac{1}{t}\Big((S(t)+I(t)+Q(t))-(S(0)+I(0)+Q(0))\Big)\\&=A-\frac{\mu_1}{t}\int^t_0 S(s)\textup{d}s-\frac{\mu_2}{t} \int^t_0I(s)\textup{d}s-\frac{\mu_3}{t}\int^t_0Q(s)\textup{d}s+\frac{\sigma_1}{t}\int^t_0 S(s) \textup{d}\mathcal{W}_1(s)+\frac{1}{t}\int^t_0\int_Z \eta_1(u)S(s^{-})\widetilde{\mathcal{N}}(\textup{d}s,\textup{d}u)\nonumber\\&\;\;\;+\frac{\sigma_2}{t}\int^s_0 I(s) \textup{d}\mathcal{W}_2(s)+\frac{1}{t}\int^t_0\int_Z \eta_2(u)I(s^{-})\widetilde{\mathcal{N}}(\textup{d}s,\textup{d}u)+\frac{\sigma_3}{t}\int^t_0 Q(s) \textup{d}\mathcal{W}_3(s)+\frac{1}{t}\int^t_0\int_Z \eta_3(u)Q(s^{-})\widetilde{\mathcal{N}}(\textup{d}s,\textup{d}u).
\end{align*}
Taking the integration for the third equation of model (\ref{s2}) yields
\begin{align}
Q(t)-Q(0)=\delta\int^t_0I(s)\textup{d}s-(\mu_3+k)\int^t_0Q(s)\textup{d}s+\sigma_3\int^t_0 Q(s) \textup{d}\mathcal{W}_3(s)+\int^t_0\int_Z \eta_3(u)Q(t^{-})\widetilde{\mathcal{N}}(\textup{d}s,\textup{d}u).
\label{take}
\end{align}
Dividing $t$ on both sides of equation (\ref{take}), we have
\begin{align*}
\frac{1}{t}\int^t_0Q(s)\textup{d}s&=\frac{\delta}{(\mu_3+k)}\frac{1}{t}\int^t_0I(s)\textup{d}s+\frac{\sigma_3}{(\mu_3+k)}\frac{1}{t}\int^t_0Q(s)\textup{d}\mathcal{W}_3(s)\\&\;\;\;+\frac{1}{(\mu_3+k)}\int^t_0\int_Z \eta_3(u)Q(s^{-})\widetilde{\mathcal{N}}(\textup{d}s,\textup{d}u)-\frac{1}{(\mu_3+k)t}(Q(t)-Q(0)).
\end{align*}
Then, one can obtain that 
\begin{align}
\frac{1}{t}\int^t_0S(s)\textup{d}s=\frac{A}{\mu_1}-\frac{1}{t}\Bigg(\frac{\mu_2}{\mu_1}+\frac{\delta \mu_3}{\mu_1(\mu_3+k)}\Bigg)\int^t_0I(s)\textup{d}s+\Phi_1(t),
\label{S(t)}
\end{align}
where
\begin{align*}
\Phi_1(t)&=\frac{\sigma_3 \mu_3}{\mu_1(\mu_3+k)t}\int^t_0Q(s)\textup{d}\mathcal{W}_3(s)+\frac{\mu_3}{\mu_1(\mu_3+k)t}\int^t_0\int_Z \eta_3(u)Q(s^{-})\widetilde{\mathcal{N}}(\textup{d}s,\textup{d}u)-\frac{1}{(\mu_3+k)t}(Q(t)-Q(0))\\&\;\;\;+\frac{\sigma_1}{\mu_1t}\int^t_0 S(s) \textup{d}\mathcal{W}_1(s)+\frac{1}{\mu_1t}\int^t_0\int_Z \eta_1(u)S(s^{-})\widetilde{\mathcal{N}}(\textup{d}s,\textup{d}u)+\frac{\sigma_2}{\mu_1t}\int^t_0 I(s) \textup{d}\mathcal{W}_2(s)\nonumber\\&\;\;\;+\frac{1}{\mu_1t}\int^t_0\int_Z \eta_2(u)I(s^{-})\widetilde{\mathcal{N}}(\textup{d}s,\textup{d}u)+\frac{\sigma_3}{\mu_1t}\int^t_0 Q(s) \textup{d}\mathcal{W}_3(s)+\frac{1}{\mu_1t}\int^t_0\int_Z \eta_3(u)Q(s^{-})\widetilde{\mathcal{N}}(\textup{d}s,\textup{d}u)\\&\;\;\;-\frac{1}{\mu_1t}\Big((S(t)+I(t)+Q(t))-(S(0)+I(0)+Q(0))\Big).
\end{align*}
Applying Itô’s formula to the second equation of (\ref{s2}), we get
\begin{align}
\text{d}\ln I(t)&=\Bigg(\beta S(t)-(\mu_2+\delta+\gamma)-\frac{\sigma_2^2}{2}-\frac{\sigma_{\beta}^2}{2}S^2(t)-\int_Z \eta_2(u)-\ln(1+\eta_2(u))\nu(\textup{d}u)\Bigg)\textup{d}t\nonumber\\&\;\;\;+\sigma_2\textup{d}\mathcal{W}_2(t)+\int_Z \ln(1+\eta_2(u))\widetilde{\mathcal{N}}(\textup{d}t,\textup{d}u)+\sigma_{\beta} S(t)\textup{d}\mathcal{W}_{\beta}(t).
\label{205}
\end{align}
Integrating (\ref{205}) from $0$ to $t$ and then dividing $t$ on both sides, we have
\begin{align*}
\frac{1}{t}(\ln I(t)-\ln I(0))&=\frac{\beta}{t}\int^t_0 S(s)\textup{d}s-(\mu_2+\delta+\gamma)-\frac{\sigma_2^2}{2}-\frac{\sigma_{\beta}^2}{2t}\int^t_0S^2(s)\textup{d}s-\int_Z \eta_2(u)-\ln(1+\eta_2(u))\nu(\textup{d}u)\nonumber\\&\;\;\;+\sigma_2\frac{W_2(t)}{t}+\frac{1}{t}\int^t_0\int_Z \ln(1+\eta_2(u))\widetilde{\mathcal{N}}(\textup{d}s,\textup{d}u)+\frac{\sigma_{\beta}}{t} \int^t_0S(s)\textup{d}\mathcal{W}_{\beta}(s).
\end{align*}
From (\ref{S(t)}), we get
\begin{align*}
\frac{1}{t}(\ln I(t)-\ln I(0))&=\frac{\beta A}{\mu_1}-\frac{\beta}{t}\Bigg(\frac{\mu_2}{\mu_1}+\frac{\delta \mu_3}{\mu_1(\mu_3+k)}\Bigg)\int^t_0I(s)\textup{d}s+\beta\Phi_1(t)-(\mu_2+\delta+\gamma)\\&\;\;\;-\frac{\sigma_2^2}{2}-\frac{\sigma_{\beta}^2}{2}\int^t_0S^2(s)\textup{d}s-\int_Z \eta_2(u)-\ln(1+\eta_2(u))\nu(\textup{d}u)\nonumber\\&\;\;\;+\sigma_2\frac{W_2(t)}{t}+\frac{1}{t}\int^t_0\int_Z \ln(1+\eta_2(u))\widetilde{\mathcal{N}}(\textup{d}s,\textup{d}u)+\frac{\sigma_{\beta}}{t} \int^t_0S(s)\textup{d}\mathcal{W}_{\beta}(s).
\end{align*}
Since $S(t)\leq X(t)$ a.s., we obtain
\begin{align*}
\frac{1}{t}(\ln I(t)-\ln I(0))&\geq \frac{\beta A}{\mu_1}-\frac{\beta}{t}\Bigg(\frac{\mu_2}{\mu_1}+\frac{\delta \mu_3}{\mu_1(\mu_3+k)}\Bigg)\int^t_0I(s)\textup{d}s+\beta\phi_1(t)-(\mu_2+\delta+\gamma)\\&\;\;\;-\frac{\sigma_2^2}{2}-\frac{\sigma_{\beta}^2}{2}\int^t_0X^2(s)\textup{d}s-\int_Z \eta_2(u)-\ln(1+\eta_2(u))\nu(\textup{d}u)\nonumber\\&\;\;\;+\sigma_2\frac{W_2(t)}{t}+\frac{1}{t}\int^t_0\int_Z \ln(1+\eta_2(u))\widetilde{\mathcal{N}}(\textup{d}s,\textup{d}u)+\frac{\sigma_{\beta}}{t} \int^t_0S(s)\textup{d}\mathcal{W}_{\beta}(s).
\end{align*}
Hence, we further have
\begin{align*}
\frac{\beta}{t}\Bigg(\frac{\mu_2}{\mu_1}+\frac{\delta \mu_3}{\mu_1(\mu_3+k)}\Bigg)\int^t_0I(s)\textup{d}s&\geq -\frac{1}{t}(\ln I(t)-\ln I(0))+\frac{\beta A}{\mu_1}+\beta\phi_1(t)-(\mu_2+\delta+\gamma)\\&\;\;\;-\frac{\sigma_2^2}{2}-\frac{\sigma_{\beta}^2}{2}\int^t_0X^2(s)\textup{d}s-\int_Z \eta_2(u)-\ln(1+\eta_2(u))\nu(\textup{d}u)\nonumber\\&\;\;\;+\sigma_2\frac{W_2(t)}{t}+\frac{1}{t}\int^t_0\int_Z \ln(1+\eta_2(u))\widetilde{\mathcal{N}}(\textup{d}s,\textup{d}u)+\frac{\sigma_{\beta}}{t} \int^t_0S(s)\textup{d}\mathcal{W}_{\beta}(s).
\end{align*}
By assumption \textbf{($A_3$)}, Lemmas \ref{lem1m} - \ref{lemmas}, and the large number theorem for martingales, we can easily verify that
\begin{align*}
&\underset{t\to\infty }{\liminf}\frac{1}{ t}\int^t_0I(s)\textup{d}s\geq\frac{1}{\beta} \Bigg(\frac{\mu_2}{\mu_1}+\frac{\delta \mu_3}{\mu_1(\mu_3+k)}\Bigg)^{-1}\Big(\mu_2+\delta+\gamma+\frac{\sigma_2^2}{2}\Big)(\mathcal{R}^s_0-1)>0\hspace{0.2cm}\mbox{a.s.}
\end{align*}
This shows that the system (\ref{s2}) is persistent in the mean with probability one. This completes the proof.
\end{proof}
\subsection{The extinction of the disease}
Now, we will give the result on the extinction of the disease. Define
\begin{align*}
\mathcal{\hat{R}}^s_0&=\Big(\mu_2+\delta+\gamma+\frac{\sigma^2_2}{2}\Big)^{-1}\bigg(\frac{\beta A}{\mu_1}-\frac{\sigma_{\beta}^2A^2}{2\mu_1^2}-\int_Z\eta_2(u)-\ln(1+\eta_2(u))\nu(\textup{d}u)\bigg).
\end{align*}
\begin{thm}
Let $Y(t)$ be the solution of system (\ref{s2}) with initial value $Y(0)\in\R^3_+$.\\ If 
\begin{align}
&\hat{\mathcal{R}}^s_0<1\hspace{0.2cm}\mbox{and} \hspace{0.2cm} \sigma_{\beta}^2\leq \frac{\mu_1\beta}{A},&
\label{cond1}
\end{align}
or
\begin{align}
&\frac{\beta^2}{2\sigma_{\beta}^2}-\Big(\mu_2+\delta+\gamma+\frac{\sigma^2_2}{2}\Big)-\int_Z \eta_2(u)-\ln(1+\eta_2(u))\nu(\textup{d}u)<0,&
\label{cond2}
\end{align}
then the disease dies out exponentially with probability one. That is to say,
\begin{align}
\underset{t\to \infty}{\lim \sup}\frac{\ln I(t)}{t}<0\hspace*{0.3cm}\mbox{a.s.}
\label{ext}
\end{align}
\label{thm2}
\end{thm}
\begin{proof}
By Itô’s formula for all $t \geq0$, we have
\begin{align}
\text{d}\ln I(t)&=\Bigg(\beta S(t)-(\mu_2+\delta+\gamma)-\frac{\sigma_2^2}{2}-\frac{\sigma_{\beta}^2}{2}S^2(t)-\int_Z \eta_2(u)-\ln(1+\eta_2(u))\nu(\textup{d}u)\Bigg)\textup{d}t\nonumber\\&\;\;\;+\sigma_2\textup{d}\mathcal{W}_2(t)+\int_Z \ln(1+\eta_2(u))\widetilde{\mathcal{N}}(\textup{d}t,\textup{d}u)+\sigma_{\beta} S(t)\textup{d}\mathcal{W}_{\beta}(t).
\label{25}
\end{align}
Integrating (\ref{25}) from $0$ to $t$ and then dividing $t$ on both sides, we get
\begin{align}
\frac{\ln I(t)}{t}= \frac{\beta}{t}\int^t_0S(s)\textup{d}s-\Big(\mu_2+\delta+\gamma+\frac{\sigma^2_2}{2}\Big)-\int_Z \eta_2(u)-\ln(1+\eta_2(u))\nu(\textup{d}u)-\frac{\sigma_{\beta}^2}{2t}\int^t_0S^2(s)\textup{d}s+\Phi_2(t),
\label{28}
\end{align}
where
\begin{align*}
\Phi_2(t)=\frac{\sigma_{\beta}}{t}\int^t_0 S(s)\textup{d}\mathcal{W}_{\beta}(s)-\frac{\sigma_2W_2(t)}{t}+\frac{1}{t}\int^t_0\int_Z \ln(1+\eta_2(u))\widetilde{\mathcal{N}}(\textup{d}s,\textup{d}u)-\frac{\ln I(0)}{t}.
\end{align*}
Obviously, we know that
\begin{align*}
\frac{1}{t}\int^t_0S^2(s)\textup{d}s\geq \Big(\frac{1}{t}\int^t_0S(s)\textup{d}s\Big)^2.
\end{align*}
Therefore, from (\ref{S(t)}), we derive
\begin{align*}
\frac{\ln I(t)}{t}&\leq \frac{\beta}{t}\int^t_0S(s)\textup{d}s-\Big(\mu_2+\delta+\gamma+\frac{\sigma^2_2}{2}\Big)-\int_Z \eta_2(u)-\ln(1+\eta_2(u))\nu(\textup{d}u)-\frac{\sigma_{\beta}^2}{2}\Big(\frac{1}{t}\int^t_0S(s)\textup{d}s\Big)^2+\Phi_2(t)\\&= \beta \Bigg(\frac{A}{\mu_1}-\frac{1}{t}\Bigg(\frac{\mu_2}{\mu_1}+\frac{\delta \mu_3}{\mu_1(\mu_3+k)}\Bigg)\int^t_0I(s)\textup{d}s+\phi_1(t)\Bigg)-\Big(\mu_2+\delta+\gamma+\frac{\sigma^2_2}{2}\Big)-\int_Z \eta_2(u)-\ln(1+\eta_2(u))\nu(\textup{d}u)\\&\;\;\; -\frac{\sigma_{\beta}^2}{2}\Bigg(\frac{A}{\mu_1}-\frac{1}{t}\Bigg(\frac{\mu_2}{\mu_1}+\frac{\delta \mu_3}{\mu_1(\mu_3+k)}\Bigg)\int^t_0I(s)\textup{d}s+\phi_1(t)\Bigg)^2+\Phi_2(t).
\end{align*}
Hence, one can see that
\begin{align}
\frac{\ln I(t)}{t}&\leq\frac{\beta A}{\mu_1}-\Big(\mu_2+\delta+\gamma+\frac{\sigma^2_2}{2}\Big)-\int_Z \eta_2(u)-\ln(1+\eta_2(u))\nu(\textup{d}u)-\frac{A^2\sigma_{\beta}^2}{2\mu_1^2}\nonumber\\&\;\;\;-\Bigg(\frac{\mu_2}{\mu_1}+\frac{\delta \mu_3}{\mu_1(\mu_3+k)}\Bigg)\bigg(\beta-\frac{A\sigma_{\beta}^2}{\mu_1}\bigg)\frac{1}{t}\int^t_0I(s)\textup{d}s\nonumber\\&\;\;\;
-\frac{\sigma_{\beta}^2}{2t^2}\Bigg(\bigg(\frac{\mu_2}{\mu_1}+\frac{\delta \mu_3}{\mu_1(\mu_3+k)}\bigg)\int^t_0I(s)\textup{d}s\Bigg)^2+\Phi_2(t)+\Phi_3(t),
\label{32}
\end{align}
where
\begin{align*}
\Phi_3(t)=\beta\Phi_1(t)-\frac{\sigma^2_{\beta}}{2}\Phi^2_1(t)-\frac{\sigma_{\beta}^2A\Phi_1(t)}{\mu_1}+\sigma_{\beta}^2\Phi_1(t)\bigg(\frac{\mu_2}{\mu_1}+\frac{\delta \mu_3}{\mu_1(\mu_3+k)}\bigg)\int^t_0I(s)\textup{d}s.
\end{align*}
Based on Lemma \ref{lem1m}, one has
\begin{align*}
\underset{t\to \infty}{\lim }\frac{\Phi_2(t)}{t}=\underset{t\to \infty}{\lim }\frac{\Phi_3(t)}{t}=0\hspace*{0.3cm}\mbox{a.s.}
\end{align*}
Taking the superior limit on both sides of (\ref{32}), then by condition (\ref{cond1}), we arrive at
\begin{align*}
\underset{t\to \infty}{\lim \sup}\frac{\ln I(t)}{t}&\leq \Big(\mu_2+\delta+\gamma+\frac{\sigma^2_2}{2}\Big)\Big(\hat{\mathcal{R}}^s_0-1\Big)<0\hspace*{0.3cm}\mbox{a.s.}
\end{align*}
Now, from (\ref{28}), we have
\begin{align*}
\frac{\ln I(t)}{t}&=\frac{\beta}{t}\int^t_0S(s)\textup{d}s-\Big(\mu_2+\delta+\gamma+\frac{\sigma^2_2}{2}\Big)-\int_Z \eta_2(u)-\ln(1+\eta_2(u))\nu(\textup{d}u)-\frac{\sigma_{\beta}^2}{2t}\int^t_0S^2(s)\textup{d}s+\Phi_2(t)\\&=\frac{\beta^2}{2\sigma_{\beta}^2}-\Big(\mu_2+\delta+\gamma+\frac{\sigma^2_2}{2}\Big)-\int_Z \eta_2(u)-\ln(1+\eta_2(u))\nu(\textup{d}u)-\frac{\sigma_{\beta}^2}{2}\frac{1}{t}\int^t_0\bigg(S(s)\textup{d}s-\frac{\beta}{\sigma_{\beta}^2}\bigg)^2ds+\Phi_2(t)\\&\leq 
\frac{\beta^2}{2\sigma_{\beta}^2}-\Big(\mu_2+\delta+\gamma+\frac{\sigma^2_2}{2}\Big)-\int_Z \eta_2(u)-\ln(1+\eta_2(u))\nu(\textup{d}u))+\Phi_2(t).
\end{align*}
By the large number theorem for martingales, Lemma \ref{lem1m} and the condition (\ref{cond2}), our desired result (\ref{ext}) holds true. This completes the proof.
\end{proof}
\section{Examples}
In this section, we will validate our theoretical results with the help of numerical simulation examples taking parameters from the theoretical data mentioned in the Table \ref{value}. We numerically simulate the solution of system (\ref{s2}) with the initial values $(S(0), I(0), Q(0)) = (0.5, 0.3, 0.1)$. The unit of time is one day.
\begin{figure}[!htb]\centering
\begin{center}$
\begin{array}{cc}
\includegraphics[width=3.5in]{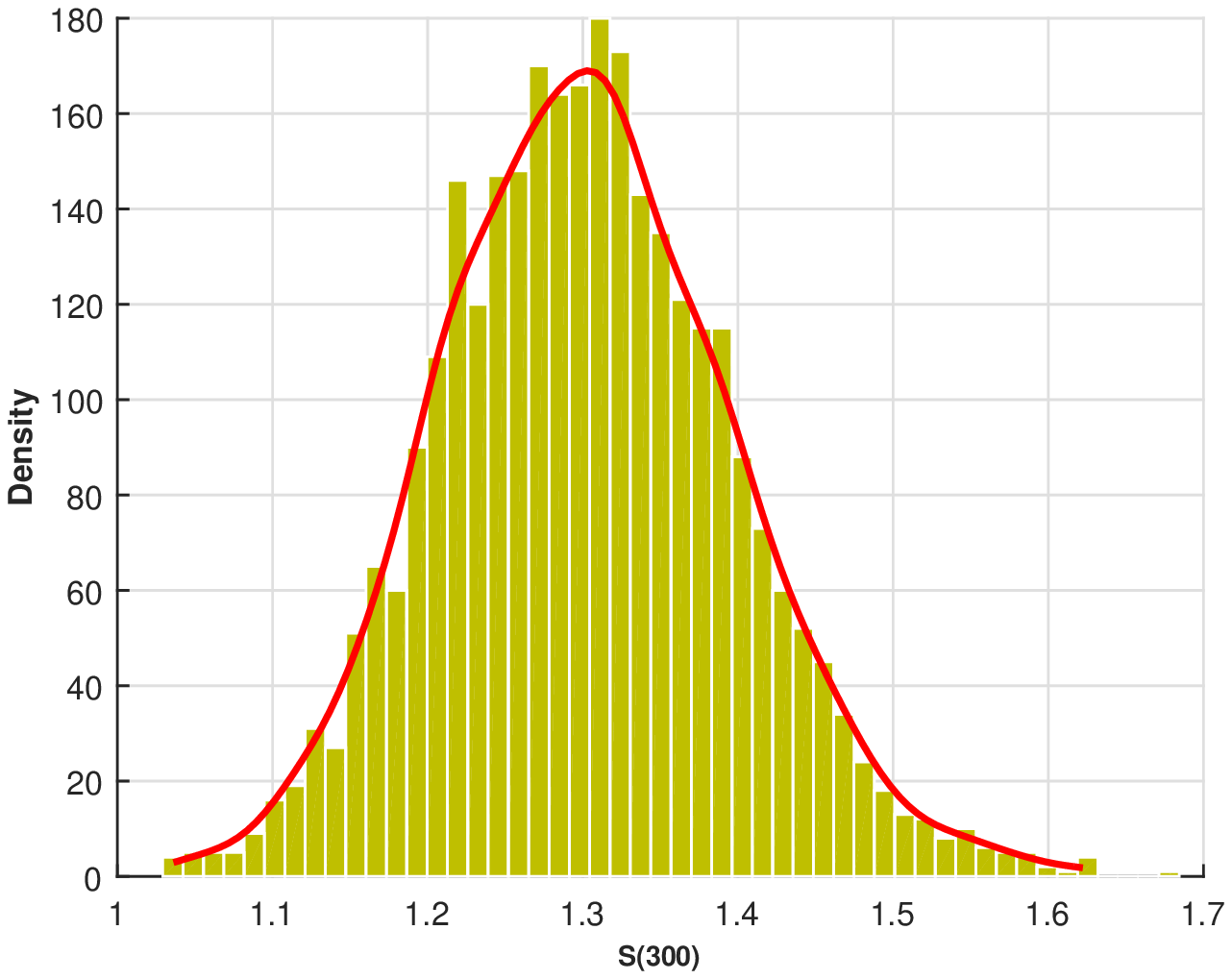} &
\includegraphics[width=3.5in]{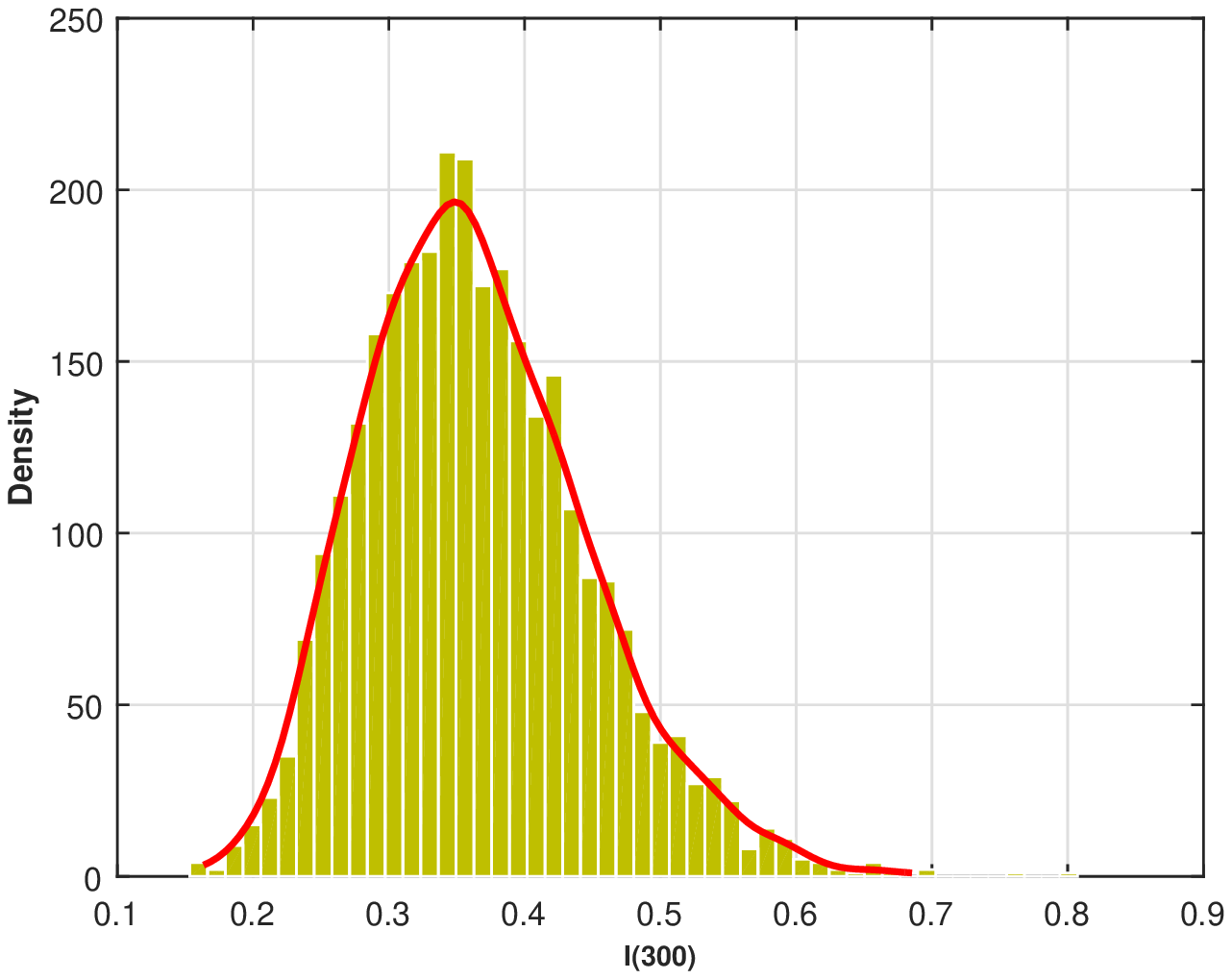}
\end{array}$
$\begin{array}{c}
\includegraphics[width=3.5in]{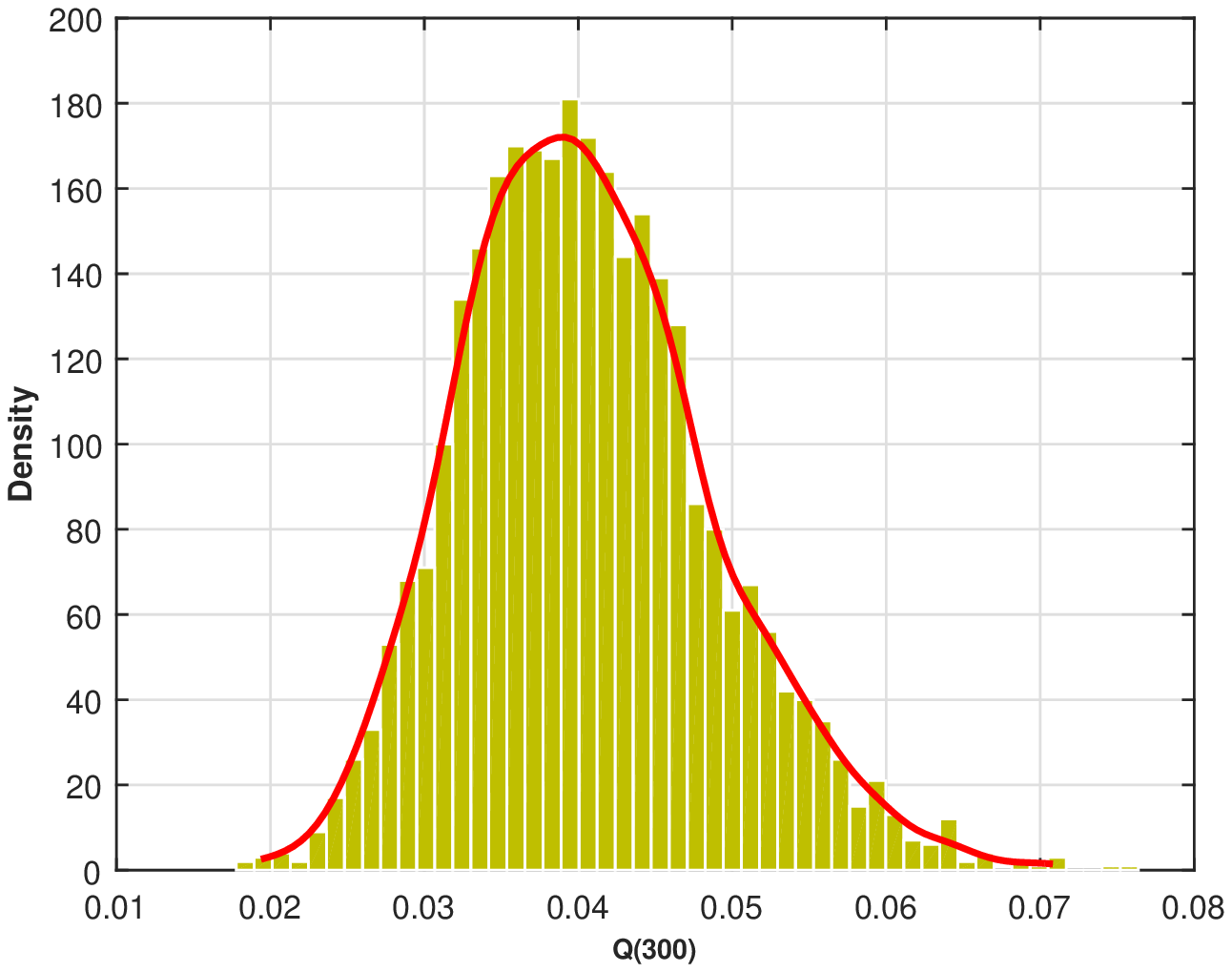}
\end{array}$
\end{center}
\caption{Histogram of the probability density function for $S$, $I$, and $Q$ population at $t=300$ for the stochastic model (\ref{s2}), the smoothed curves are the probability density functions of $S(t)$, $I(t)$ and $Q(t)$, respectively.}
\label{fig1}
\end{figure}
\begin{figure}[!htb]\centering
\begin{center}
$\begin{array}{c}
\includegraphics[width=3.5in]{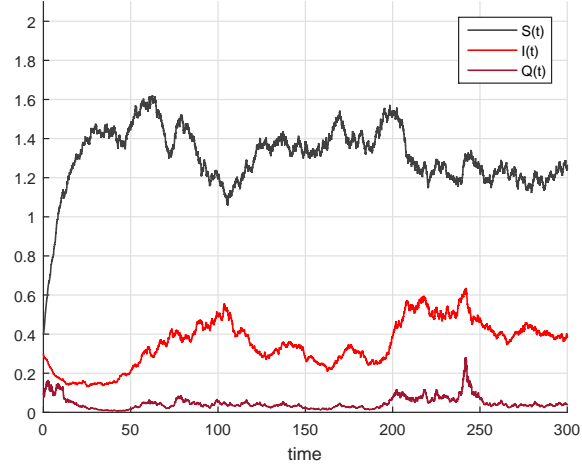}
\end{array}$
\end{center}
\caption{The paths of $S(t)$, $I(t)$ and $Q(t)$ for the stochastic model (\ref{s2}) with initial values $(S(0), I(0), Q(0)) =(0.5, 0.3, 0.1)$.}
\label{fig21}
\end{figure}
\begin{figure*}[!htb]
\subcaptionbox{}
{\includegraphics[width=3.5in]{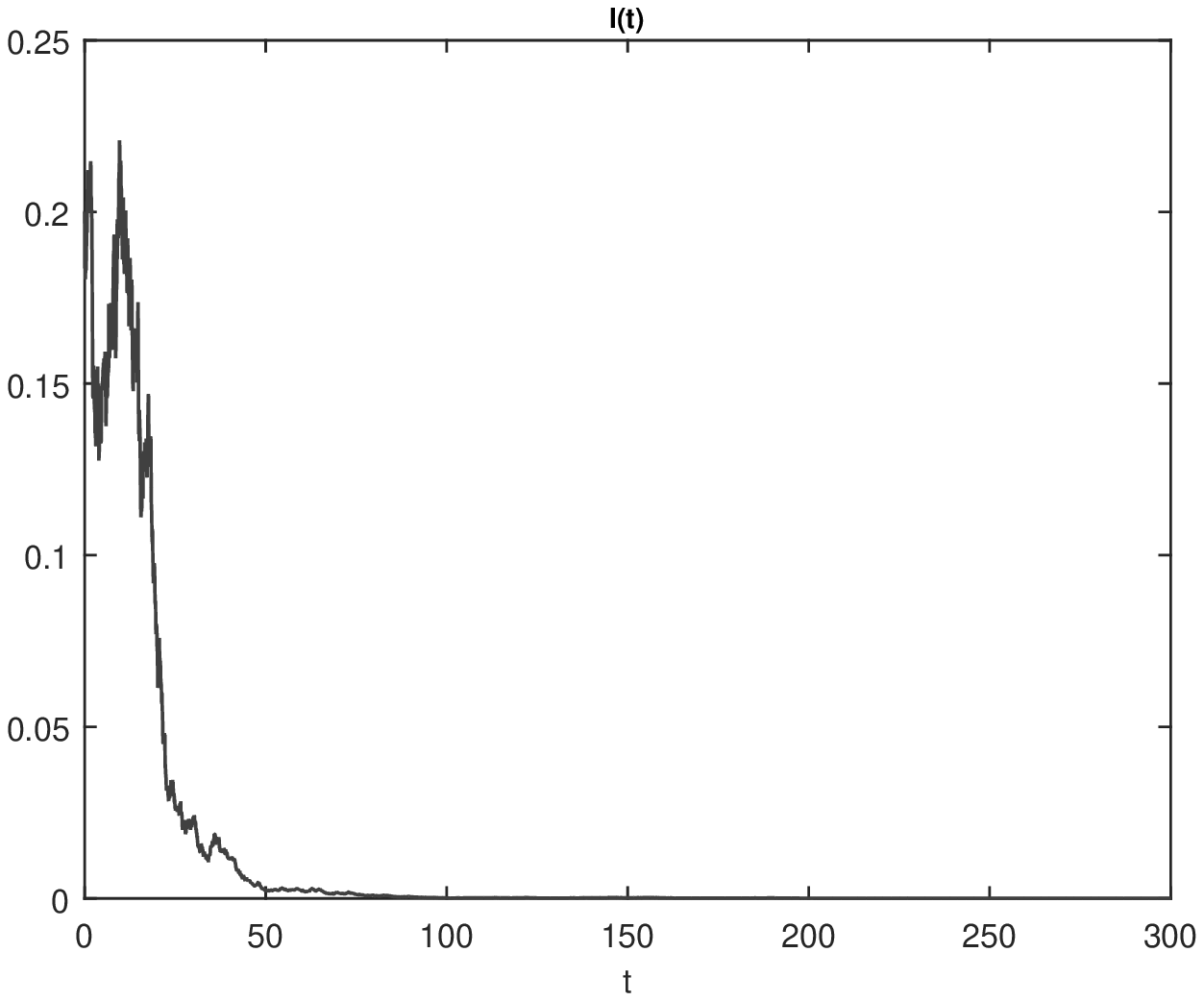} } 
\subcaptionbox{}
{\includegraphics[width=3.5in]{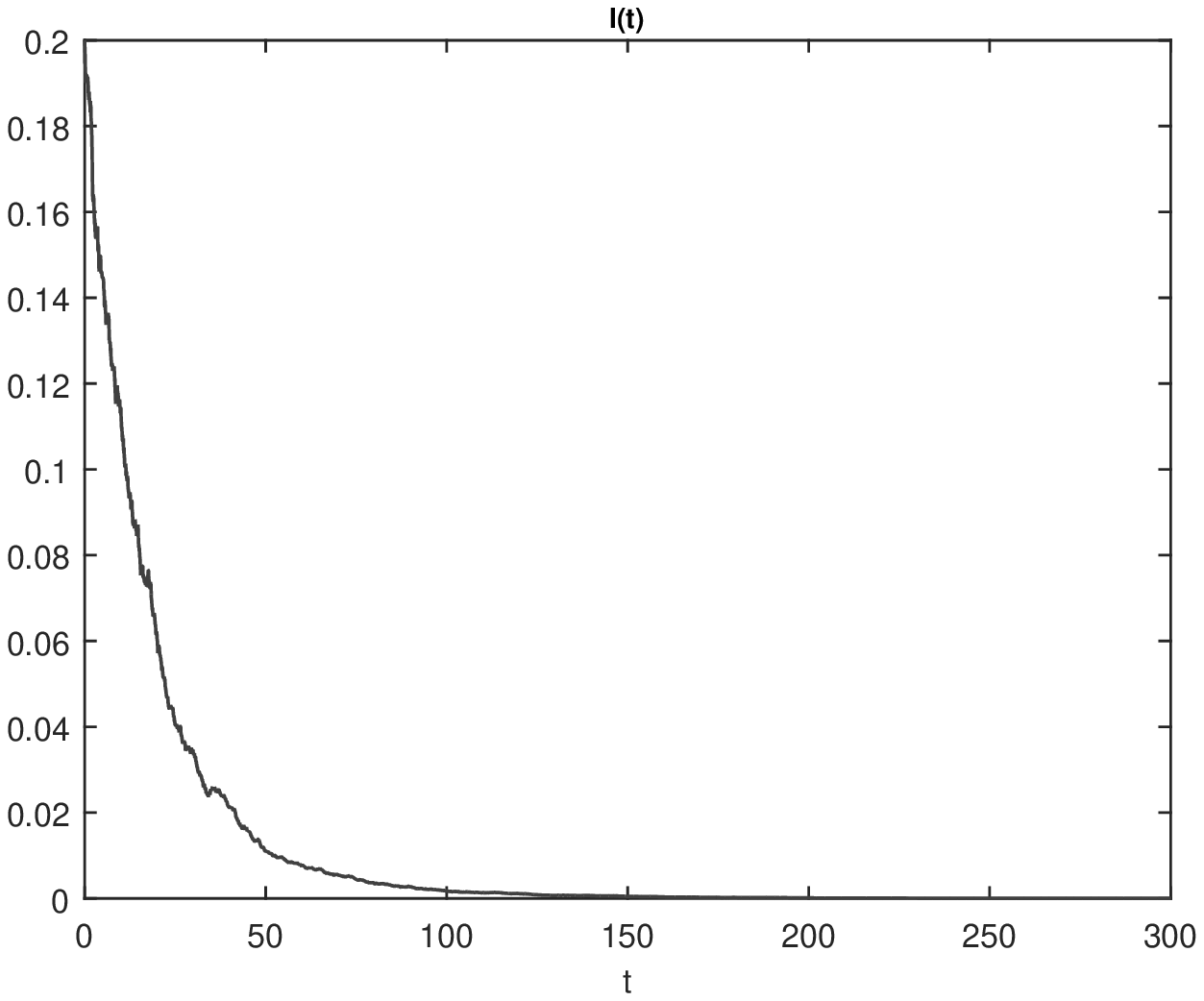}} 
\caption{The numerical simulation of $I(t)$ in the system (\ref{s2}).}
\label{fig2}
\end{figure*}
\begin{center}
\begin{tabular}{llll}
\hline
Parameters  \hspace*{0.5cm}&Description  \hspace*{0.5cm}&Value  \\
\hline
$A$ & The recruitment rate & 0.1 \\
$\mu_1$ & The natural mortality rate & 0.05\\
$\mu_2$ & The mortality rate of $I$ & 0.09\\
$\mu_3$ & The mortality rate of $Q$ & 0.052\\
$\beta$ & The transmission rate &0.075 \\
$\delta$ & The isolation rate & 0.03\\
$\gamma$ & The recovered rate of $I$ &0.01 \\
$k$ &The recovered rate of $Q$ &0.04 \\
\hline
\end{tabular}
\captionof{table}{Some theoretical parameter values of the model (\ref{s2}).}
\label{value}
\end{center}
\begin{ex}
We have chosen the stochastic fluctuations intensities $\sigma_1=0.01$, $\sigma_2=0.03$, $\sigma_3=0.07$ and $\sigma_{\beta}=0.02$. Furthermore, we assume that $\eta_1(u)=0.01$, $\eta_2(u)=0.02$, $\eta_3(u)=0.05$, $Z=(0,\infty)$ and $\nu(Z)=1$. Then, $\mathcal{R}_0^s=1.1756
>1$. From Figure \ref{fig1}, we show the existence of the unique stationary distributions for $S(t)$, $I(t)$ and $Q(t)$ of the model (\ref{s2}) at $t = 300$, where the smooth curves are the probability density functions of $S(t)$, $I(t)$ and $Q(t)$, respectively. It can be obviously observed that the solution of the stochastic model (\ref{s2}) persists in the mean (see Figure \ref{fig21}). 
\end{ex}
\begin{ex}
Now, we choose the white noise intensities $\sigma_2=0.12$ and $\sigma_{\beta}=0.1$ to ensure that the condition (\ref{cond2}) of theorem (\ref{thm2}) is satisfied. We can conclude that for any initial value, $I(t)$ obeys
\begin{align*}
\underset{t\to \infty}{\lim \sup}\frac{1}{t}\ln\frac{I(t)}{I(0)}&\leq \frac{\beta^2}{2\sigma_{\beta}^2}-\Big(\mu_2+\delta+\gamma+\frac{\sigma^2_2}{2}\Big)-\int_Z \eta_2(u)-\ln(1+\eta_2(u))\nu(\textup{d}u),\\&=-0.1374<0 \;\;\; \mbox{a.s.}
\end{align*}
That is, $I(t)$ will tend to zero exponentially with probability one (see Figure \ref{fig2} (a)). To verify that the condition (\ref{cond1}) is satisfied, we change $\sigma_2$ to $0.01$, $\sigma_{\beta}$ to $0.02$ and $\beta$ to $0.05$ and keep other parameters unchanged. Then we have
\begin{align*}
\mathcal{\hat{R}}^s_0&=\Big(\mu_2+\delta+\gamma+\frac{\sigma^2_2}{2}\Big)^{-1}\bigg(\frac{\beta A}{\mu_1}-\frac{\sigma_{\beta}^2A^2}{2\mu_1^2}-\int_Z\eta_2(u)-\ln(1+\eta_2(u))\nu(\textup{d}u)\bigg)= 0.7650<1,
\end{align*}
and
\begin{align*}
\sigma_{\beta}^2-\frac{\mu_1\beta}{A}=-0.0249<0.
\end{align*}
Therefore, the condition (\ref{cond1}) of Theorem \ref{thm2} is satisfied. We can conclude that for any initial value, $I(t)$ obeys
\begin{align*}
&\underset{t\to \infty}{\lim \sup}\frac{1}{t}\ln\frac{I(t)}{I(0)}\leq \Big(\mu_2+\delta+\gamma+\frac{\sigma^2_2}{2}\Big)\Big(\hat{\mathcal{R}}^s_0-1\Big)=-0.0306<0 \;\;\; \mbox{a.s.}
\end{align*}
That is, $I(t)$ will tend to zero exponentially with probability one (see Figure \ref{fig2} (b)).
\end{ex}
\section*{Conclusion}
In this study, we proposed a new version of a perturbed SIS epidemiological model with a quarantine strategy. This model simultaneously takes into account random transmission and the effects of jumps. We have addressed possible scenarios of the pandemic spread during unforeseen climate changes or environmental shocks. Compared with the existing literature, the novelty of our study manifested in new analysis techniques and improvements which are summarized in the following items:\\
\begin{itemize}
\item[$\bullet$] Our paper is distinguished from previous works \cite{57,166,lev1,lev2,lev3} by improving the majorization of the following quantity
\begin{align*}
\int_ZN^{np}(t)\left[\Big(1+\frac{\tilde{X}}{N}\Big)^{np}-1-np\frac{\tilde{X}}{N}\right]\nu(\textup{d}u),
\end{align*}
which raises the optimality of our calculus and results.  
\item[$\bullet$] Our results in Lemmas \ref{lem22} and \ref{lem1m} provide an extended and generalized version of classical lemmas 3.3 and 3.4 presented in \cite{166} which are widely used in the literature.
\item[$\bullet$] Our study provides an improved threshold 
\begin{align*}
\mathcal{R}^s_0&=\Big(\mu_2+\delta+\gamma+\frac{\sigma_2^2}{2}\Big)^{-1}\left(\frac{\beta A}{\mu_1}-\frac{A^2\sigma_{\beta}^2}{\mu_1\chi}-\int_Z \eta_2(u)-\ln(1+\eta_2(u))\nu(\textup{d}u)\right),
\end{align*}
by taking into consideration the Remark \ref{remm}. This parameter is a sufficient condition for the existence of a unique ergodic stationary distribution and persistence of the disease under some assumptions. The last two asymptotic properties are proven in Theorems \ref{thm1} and \ref{persis}, by using a new approach based on Lemma \ref{lemmas} and the mutually exclusive possibilities lemma \ref{poss}. 
\item[$\bullet$] Our study offers an alternative method to the gap mentioned in (Theorem 2.2, \cite{58}). Without using the explicit formula of the distribution stationary  $\mu(\cdot)$ of $X$ (which still up to now unknown), we gave the expression of the ergodicity and persistence threshold.  
\item[$\bullet$] For the case of non-persistence, in Theorem \ref{thm2}, we proved that the following parameter
\begin{align*}
\mathcal{\hat{R}}^s_0&=\Big(\mu_2+\delta+\gamma+\frac{\sigma^2_2}{2}\Big)^{-1}\bigg(\frac{\beta A}{\mu_1}-\frac{A^2\sigma_{\beta}^2}{2\mu_1^2}-\int_Z\eta_2(u)-\ln(1+\eta_2(u))\nu(\textup{d}u)\bigg),
\end{align*}
is a sufficient conditions for the disappearance of the disease. 
\end{itemize}
\vspace{0.3cm}
Eventually, we point out that the obtained results extend and generalize many previous works (for example, \cite {sis11, sis12, sis13, sis14, sis2}), by analyzing the dynamics of the SIQS epidemic models with two disturbances. We believe that our article can be a rich basis for future studies. 
%
\bibliographystyle{plain}
\bibliography{double}
\end{document}